\documentclass{article}
\usepackage{style}
\usepackage{mymacros}
\usepackage{citationscheme}
\usepackage{quiver}
\title{On the \texorpdfstring{$K$}{K}-theory of algebraic tori}

\author{Qingyuan Bai \and Shachar Carmeli \and Branko Juran\and Florian Riedel}

\begin{document}

\maketitle
\begin{abstract}
    Given an algebraic torus $T$ over a field $F$, its lattice of characters 
    $\Lambda$ gives rise to a topological torus $\topTor{T}=\Lambda_{\mathbb R}/\Lambda$ 
    with a continuous action of the absolute Galois group $G$.
    We construct a natural equivalence between the algebraic $K$-theory $K_{\ast}(T)$ and the
    equivariant homology $H^{G}_{\ast}(\topTor{T};K_G(F))$ of the topological torus $\topTor{T}$ with coefficients in the
    $G$-equivariant $K$-theory of $F$. 
    This generalizes a computation of $K_0(T)$ due to Merkurjev and Panin.
    We obtain this equivalence by analyzing the motive $\KGL_{\baf}^{T}$ in the stable motivic category $\SH{\baf}$ of Voevodsky and Morel, where $\KGL_{\baf}$ is the motivic spectrum
    representing homotopy $K$-theory. We construct a natural comparison map 
    $\Four\colon \KGL_{\baf}[\deloop{\Lambda}] \to \KGL_{\baf}^{T}$ from the $\KGL_{\baf}$-homology of the \'etale delooping of $\Lambda$ to $\KGL_{\baf}^{T}$ as a special case of a motivic Fourier 
    transform and prove that it is an equivalence by using a motivic Eilenberg--Moore formula
    for classifying spaces of tori.
\end{abstract}

\tableofcontents
\newpage
\section{Introduction}

The \enquote{fundamental theorem} of algebraic $K$-theory, due to Quillen \cite{qui72},
asserts that for a field~$\baf$, the $K$-theory of the multiplicative group $\GG_{m,F}$ is given by $K_{\ast}(\GG_{m,F}) \simeq K_{\ast}(F) \oplus K_{\ast-1}(F)$.
Put differently, the $K$-groups are computed by the homology
\begin{equation}\label{eq: Quillens comp}
 K_{\ast}(\GG_{m,F})\simeq H_{\ast}(S^1; K(F))
\end{equation}
of the circle $S^1$ with coefficients in the $K$-theory spectrum of $F$.

We generalize this formula to algebraic tori which are non-split. 
Before we set up the statement, let us first be explicit about what our objects of study are.

\begin{ter}
Let $F$ be a field and $T$ a commutative group scheme over $F$. Then $T$ is called an
\tdef{algebraic torus} if it is isomorphic
to a split torus $T_E\simeq \GG_{m,E}^{\times n}$ for some $n\geq 1$ after base change
along a finite Galois extension $E/F$. In this situation, $n$ is called the \tdef{rank} and 
the field extension $E/F$ is called a \tdef{splitting field} of the torus $T$. 
\end{ter}

Given a rank $n$ torus $T/F$ with a choice of splitting field $E$ and Galois group 
$G=\mathrm{Gal}(E/F)$, we can attach an algebraic invariant to $T$ as follows:
The \tdef{character lattice} of $T$ is the free abelian
group of rank $n$ with $G$-action defined as
\[ \mdef{\cow{T}}\coloneq \hom(T_E, \GG_{m,E})\,.\]
Namely, a class $\chi \in \cow{T}$ is given by a character $\chi \colon T_E\to E^\times$
and $G$ acts via the action on $E$. To extend Quillen's computation 
\eqref{eq: Quillens comp} to non-split tori, we make the following replacements:

\begin{enumerate}
    \item The role of the space $S^1$ will be played by the \tdef{topological mirror}
          of $T$, namely the topological space with $G$-action given by the quotient
          space
          \[ \mdef{\topTor{T}}\coloneq (\cow{T} \otimes_{\ZZ}\RR)/\cow{T}\,.\]
    \item The role of the coefficients $K(F)$ will be played by the 
          \tdef{$G$-equivariant $K$-theory} of $F$ which we denote by \mdef{$K_G(F)$}. 
          This is a genuine $G$-spectrum representing the equivariant cohomology theory
          which, for any normal subgroup $H\subseteq G$, evaluates on the orbits as
          $H^{\ast}(G/H;K_G(F))= K_{\ast}(E^H)$.\footnote{See~\cite[\href{https://arxiv.org/pdf/1404.0108}{Ex. D.24}]{bar17} as well as \cref{defn: G equiv K theory} for a precise construction of
          $K_G(F)$.}
\end{enumerate}

Indeed, for $\t=\Gm$, the character lattice is given by $\ZZ$ with no action and we see that
$\topTor{\t}= S^1$. For a non-split example, the torus over $\R$
defined by the equation $x^2+y^2=1$ is a non-split torus of rank one and has topological mirror
given by $S^1$ equipped with the $C_2$-action induced by complex conjugation on $S^1 \subseteq \CC$.

With all players in place, we can state the most digestible version of our main theorem.

\begin{thmx}[\Cref{equivariant mainthm}]\label{thm: thm A intro}
Let $\t/F$ be an algebraic torus with splitting field $E$ and Galois group
$G=\mathrm{Gal}(E/F)$.
We have a natural equivalence of graded rings 
\[H_\ast^G(\topTor{\t};\eqK(F)) \xrightarrow{\sim} K_\ast(\t)\]
between the $G$-equivariant homology of the topological mirror $\topTor{\t}$ with coefficients
in $\eqK(F)$ and the algebraic $K$-theory of $\t$.
\end{thmx}

Here, the ring structure on $H_{\ast}^G(\topTor{\t};\eqK(F))$ is given by 
convolution, i.e.~induced by the natural $\GalG$-equivariant Lie group structure on the torus $\topTor{\t}$. 

This provides a method for computing $K_{\ast}(T)$ in terms of the $K$-theory
of all intermediate extensions $F\subseteq L \subseteq E$ via any $G$-equivariant cell
structure on the topological torus $\topTor{T}$. 
Moreover, for $\ast = 0$, the explicit formula for Bredon homology combined with Hilbert 90,
recovers the following result of Merkurjev and Panin~\cite{mp97}, therein
stated in terms of generators and relations.

\begin{cor}[\cref{K_0 computation}] \label{corintro}
    Let $\t$ be an algebraic torus defined over a field $F$ with splitting field $E$ and
    Galois group $G=\mathrm{Gal}(E/F)$. There is natural equivalence 
    \[ H_0^G(\topTor{\t};\underline{\Z}) \xrightarrow{\sim} K_0(T) \] 
    between the $0$-th integral Bredon homology of $\topTor{\t}$ and the algebraic $K$-theory group $K_0(T)$. 
\end{cor}

Although we write it in the equivariant language, the proof 
of \cref{thm: thm A intro} does not proceed through equivariant homotopy theory. 
Instead, we lift everything to the land of \emph{motivic} homotopy theory and in fact prove 
a much more highly structured statement, namely \cref{thm: motivic main thm}. 
Let us switch gears to motivic homotopy theory and outline what actually goes into the proof.

\subsection*{The motivic story}

For a field $F$, we denote by \mdef{$\SH{\baf}$} the stable motivic category 
of $\baf$ in the sense of Voevodsky--Morel. This is a stable, symmetric monoidal 
$\infty$-category, which is constructed such that an object $E\in \SH{\baf}$ represents cohomology theory $H^{\ast}(-;E)$ on the category 
of smooth $\baf$-schemes $\Sm{\baf}$, given on some $X\in \Sm{\baf}$ by the (bigraded) homotopy groups 
of the motive $E^{X}\in \SH{\baf}$.
The objects appearing in \cref{thm: thm A intro} translate as follows into this setup:

\begin{enumerate}
    \item The role of the genuine $\GalG$-spectrum $K_G$ is played by the motivic
          spectrum $\mdef{\KGL_{\baf}} \in \SH{\baf}$ representing algebraic $K$-theory of 
          smooth $\baf$-schemes.
    \item For a torus $T/F$, the role of the topological mirror $\topTor{\t}$ is played
          by the \'etale delooping $\mdef{\deloop{\cow{\t}}}$ where $\cow{\t}$ is the 
          \'etale sheaf of group homomorphisms $\t\to \Gm$.
\end{enumerate}

Our goal in life is to produce a natural equivalence
\[ H_{\ast}(\deloop{\cow{\t}};\KGL_{\baf})\xrightarrow{\sim} H^{\ast}(\t;\KGL_{\baf}) = K_{\ast}(T)\]
between the motivic \emph{homology} of the \'etale delooping $\deloop{\cow{\t}}$
and the algebraic $K$-theory of~$\t$, which is a \emph{cohomology theory} applied to $\t$.
The first step then is to construct a highly structured,
natural comparison map, which we can manipulate using the machinery of motivic homotopy theory. 
We interpret it as a kind of \tdef{Fourier transform}. 
Recall that, if $R$ is an ordinary commutative ring and $G$ a finite abelian group
with Pontryagin dual $G^{\vee}=\hom(G,\QQ/\ZZ)$, the classical Fourier
transform is a map of commutative rings
\[ \Four_{\omega}\colon R[G] \too R^{G^{\vee}}\]
depending on a choice of orientation $\omega \colon \QQ/\ZZ\to R^{\times}$. Here, $R[G]$ is
the group ring construction and $R^{G^\vee}$ is the ring of functions $G^{\vee}\to R$. This tells us that, a natural first step towards writing 
down our comparison map is 
formulating a notion of \emph{duality} that intertwines a torus with the delooping of 
its character lattice.

We write $\Ettop{\baf} = \Shv_{\et}(\Sch_{\baf};\Spc)$ for the $\infty$-category of
sheaves of spaces on the big \'etale site of $\baf$. Then any commutative group scheme $A/F$ as
well as its \'etale classifying stack $\deloop{A}$ naturally refine to abelian group objects
in the category $\Ettop{\baf}$ and we make this identification implicitly from now on.
\begin{defn}
    For any abelian group object $M\in \Ab(\Ettop{\base})$ we define the
    \tdef{shifted Cartier dual} of $M$ as the internal mapping object
    \[\mdef{M^{\vee}}\coloneq \map(M,\deloop{\Gm}) \in \Ab(\Ettop{\base})\,.\]
Moreover, we define the motivic group ring construction
\[ \mdef{\KGL_{\baf}[-]}\colon \Ab(\Ettop{\base}) \too \CAlg_{\KGL_{\baf}}(\SH{\baf})\]
as the composite of the forgetful functor to Nisnevich sheaves with the 
left Kan extension of the motivic suspension spectrum functor 
$\KGL_{\baf}\otimes \Sigma^\infty_+(-)\colon \Sm{\baf}\to \Mod{\SH{\baf}}{\KGL_{\baf}}$. 
\end{defn}
If $\t$ is a torus over $\baf$, recalling that $\cow{\t}=\map(\t, \Gm)$ it is straightforward
to show that there are natural equivalences of \'etale sheaves
\[ \cow{\t}^{\vee} \simeq \deloop{\t} \quad \text{and} \quad 
\deloop{\cow{\t}}^{\vee}\simeq \t\,,\]
i.e.~the delooping of the torus is dual to its character lattice, and the torus itself
is dual to the delooping of its character lattice. 
Moreover, we can naturally identify $\deloop{\Gm}$ with the 
sheaf of spaces taking a scheme $X/F$ to the Picard groupoid $\pic(X)$. 
Under this identification, we get an \enquote{orientation} for free via
the natural map $\pic(X)\to K(X)^{\times}$. Building on this observation, we construct the
comparison map as a $K$-theoretic delooping of a categorical Fourier--Mukai transform, 
see \cref{sec:Fourier} for a detailed discussion. Concretely, we prove the following:

\begin{prop}[\cref{defn: mot fourier}] \label{prop: intro constr mot}
Let $M\in \Ab(\Ettop{F})$, such that the shifted Cartier dual $M^{\vee}$ is representable by 
a stack~\footnote{In the sense of \cref{ter: stack}.}.
There is a natural map of $\EE_\infty$-algebras 
\[ \mdef{\Four_{\rm{mot}}^{M}}\colon \KGL_{\baf}[M] \too \KGL_{\baf}^{M^{\vee}} \,\] 
in the category $\Mod{\SH{\baf}}{\KGL_{\baf}}$ which we call the \tdef{motivic Fourier transform}.
\end{prop}
Here, the right hand side is \emph{not} defined via right Kan-extension from smooth
$F$-schemes, but using the global structure of the motivic spectrum $\KGL_{\baf}$, see also
the discussion at \cref{war: completed KGL}. Concretely, if $M^{\vee}=\deloop{\GG}$
for a well-behaved, smooth, commutative group scheme $\GG$, then the motive
$\KGL_{\baf}^{\deloop{\GG}}$ represents \emph{$\GG$-equivariant} $K$-theory. In particular,
the global sections are given by the spectrum $
C^{\ast}(F;\KGL_{\baf}^{\deloop{\GG}})=K(\Perf(\deloop{\GG}))$. 
Our main theorem in the motivic world goes as follows:

\begin{thmx}[\Cref{thm: mainthm}] \label{thm: motivic main thm}
    For any torus $\t$ over a field $\baf$, the motivic Fourier transform
    \[ \Four_{\rm{mot}}^{\deloop{\cow{\t}}} \colon \KGL_{\baf}[\deloop{\cow{\t}}]
    \xrightarrow{\sim} \KGL_{\baf}^{\t}\,,\]
    is an equivalence of $\EE_\infty$-algebras in $\Mod{\SH{\baf}}{\KGL_{\baf}}$. 
\end{thmx}

The equivariant statement of \Cref{thm: thm A intro} then follows by a standard assembly
argument about the realization functor $\SH{\baf}\to \Sp_{\GalG}$ where $\GalG$ is the 
absolute Galois group of $\baf$. Having constructed the motivic Fourier transform, the 
proof of \cref{thm: motivic main thm} proceeds by first making an auxiliary computation.

For a torus $\t$, we show that the Fourier transform $\Four_{\mathrm{mot}}^{\cow{\t}}$
witnesses the well-known fact the $\infty$-category of representations $\Perf(\deloop{\t})$ is 
\emph{semi-simple}. Concretely, given a splitting field $E/F$ and representatives 
$\chi\colon \t \to E^{\times}$ of each orbit of the character lattice,
we obtain a decomposition
\[ \Perf(\deloop{\t}) \simeq \bigoplus_{\chi} \Perf(F_{\chi})\,,\]
where $F_{\chi}$ is the fixed field of the stabilizer of $\chi$. 
We check that this is equivalence is induced by categorical version 
of the motivic Fourier transform, see \cref{prop: Fourier computation QCoh},
and learn the following from an assembly argument:

\begin{prop}[\Cref{thm: Fourier computation KGL}]\label{thm: Fourier intro}
Let $\t$ be a torus over $F$ with character lattice $\cow{\t}$. The motivic
Fourier transform
\[ \Four_{\rm{mot}}^{\cow{\t}}\colon \KGL_{\baf}[\cow{\t}] \xrightarrow{\sim} \KGL_{\baf}^{\deloop{\t}}\]
is an equivalence of $\EE_\infty$-algebras in $\Mod{\SH{\baf}}{\KGL_{\baf}}$.
\end{prop}

Having shown this, we want to make full use of the naturality of the Fourier transform
to deduce \cref{thm: mainthm} from an Eilenberg--Moore formula. 
We do this using the following presentation of an arbitrary torus $\t$ in terms of
a Weil-restricted torus.

\begin{defn}\label{defn: Weil res intro}
Let $\t/F$ be a torus and $E$ a splitting field of $\t$. The
Weil restriction $\wt$ of the base change $\t_E$ back to $F$ comes equipped with a
surjective counit map $\wt \to \t$. Taking the kernel, we get a short exact sequence of tori
\[ 0 \to \kt \to \wt \to \t \to 0\]
which we call the \tdef{Weil-resolution} of $\t$ relative to $E$. 
\end{defn}

Given a Weil resolution of $\t$, we obtain a more \enquote{tame} variant of the
natural fiber sequence $\t \to \pt \to \deloop{\t}$, namely we deloop the resolution twice to 
obtain a pullback diagram of stacks of the form
\begin{equation*}
\begin{tikzcd}
	\t & \pt \\
	{\deloop{\kt}} & {\deloop{\wt}}\,.
	\arrow[from=1-1, to=1-2]
	\arrow[from=1-1, to=2-1]
	\arrow[from=1-2, to=2-2]
	\arrow[from=2-1, to=2-2]
\end{tikzcd}
\end{equation*}
By the naturality of the motivic Fourier transform $\Four_{\mathrm{mot}}$, combined with 
the auxiliary result of \cref{thm: Fourier intro} the proof of \cref{thm: motivic main thm}
reduces to the following claim.

\begin{prop}[\Cref{algEM}]\label{thm: EM intro}
    Let $\t/\baf$ be a torus and $E$ a splitting field. Then, the associated
    Weil-resolution induces an Eilenberg--Moore formula
    \[ \KGL_{\baf}^{\deloop{\kt}} \otimes_{\KGL_{\baf}^{\deloop{\wt}}} \KGL_{\baf} 
    \xrightarrow{\sim} \KGL_{\baf}^{\t} \in \CAlg_{\KGL_{\baf}}(\SH{F}).\]
\end{prop}

An analogues Eilenberg--Moore formula for Chow groups was proven by Ananyevskiy and Geldhauser \cite[\href{https://arxiv.org/pdf/2408.09390v2\#theorem.2.3.3}{Cor.~2.3.3}]{ag24}.

Although all maps are representable and $\deloop{\kt}\to \deloop{\wt}$ is smooth, 
the map $\pt \to \deloop{\wt}$ is in general not proper, so we cannot deduce 
\cref{thm: EM intro} from a smooth and proper base change formula. 
To rectify this, we show that to obtain an Eilenberg--Moore formula, it suffices 
to check that the leg $\pt \to \deloop{\wt}$ satisfies a certain strong 
cellularity condition, namely that its motive in $\SH{\deloop{\wt}}$ is \tdef{Artin--Tate}
in the sense of \cref{def:comp}. 

The final and crucial geometric input is to verify that this condition is satisfied,
which uses that we chose $\wt$ to be Weil-restricted from a split torus.
Using an inductive argument, analogous to the one used by Merkurjev and Panin \cite{mp97}
in order to prove their variant of \Cref{corintro} and the one Ananyevskiy and Geldhauser \cite{ag24} used to prove their variant of \Cref{thm: EM intro}, we show the following:

\begin{prop}[\cref{prop:point_to_BG_comp}] \label{prop:Artin-Tate-intro}
    Let $F\to E$ be a finite separable extension and let $W$ denote the Weil-restriction
    of a split torus $\Gm^{\times n}$ along $F\to E$. Then the motive of the 0-section 
    $\pt \to \deloop{\wt}$ in $\SH{\deloop{\wt}}$ is Artin--Tate.
\end{prop}

\subsection*{Outline} 
We will start by recalling background material from motivic and equivariant homotopy theory as well as parameterized category theory in \Cref{sec: preliminaries}. Afterwards, in \Cref{sec:Fourier}, we will set up the categorical and the motivic Fourier transform and prove \Cref{prop: intro constr mot} and \Cref{thm: Fourier intro}. 
\Cref{sec: Motivic EM} deals with the more geometric aspects: We will prove the Eilenberg--Moore formula of \Cref{thm: EM intro} as well as the Artin--Tate property of \Cref{prop:Artin-Tate-intro}, resulting in \Cref{thm: motivic main thm}. Finally, in \Cref{sec: Kthy top mirror}, we will discuss the relation to equivariant homotopy theory, introduce the topological mirror and deduce \Cref{thm: thm A intro}.

\subsection*{Conventions}

\begin{enumerate}
  \item We work with the theory of $(\infty,1)$-categories as developed by Lurie        in~\cite{HTT} and~\cite{HA}.
    By a category we always mean an $(\infty,1)$-category.
    \item We fix three nested Grothendieck universes. We call categories defined in the smallest one \emph{small} and categories defined in the largest one \emph{large}.
    \item We write $\Spc$ for the category of (small) spaces, $\Cat$ for the category of (small) categories, $\BigCat$ for the large category of categories and $\Sp$ for the category
         of spectra.
    \item All schemes are quasi-compact and quasi-separated.
\end{enumerate}

\subsection*{Acknowledgments}
We would like to thank Tom Bachmann, Robert Burklund, Dustin Clausen, Maxime Ramzi and Sebastian Wolf for helpful conversations related to this paper. We are also grateful to Tom Bachmann and Maxime Ramzi for comments on an earlier draft. 
All the authors were partially supported by the Danish National Research Foundation through the Copenhagen Centre for Geometry and Topology (DNRF151). 
S.C.~is supported by a research grant from the Center for New Scientists at the Weizmann Institute of Science, and would like to thank the Azrieli Foundation for their support through an Early Career Faculty Fellowship.

\section{Preliminaries}\label{sec: preliminaries} 

We will make heavy use of the machinery
of the \emph{stable motivic category} $\SH{\stack}$ of any (well behaved) stack $\stack$, 
which we review in \cref{ssec: SH of stacks}. In the case of a scheme, this category was first constructed by Vovoedsky and Morel \cite{mv99} and the generalization to the equivariant schemes is due to Hoyois \cite{hoy17}. 
We will use the version for stacks due to Khan and
Ravi~\cite{kr24}.

In \Cref{ssec: motivic-vs-equivariant}, we discuss the relation between motivic
homotopy theory over a field $\baf$ and $\GalG$-equivariant homotopy theory where
$\GalG$ is the absolute Galois group of $\baf$.

Finally, in \Cref{ssec: parametrized}, we recall some basic results and constructions from parametrized homotopy theory as developed by Martini and Wolf in~\cite{mwyoneda},~\cite{mwcolimit} and~\cite{mwpresentable}.

\subsection{Motivic homotopy theory of stacks}\label{ssec: SH of stacks}

Let us first fix what we actually mean by \emph{stack} throughout this paper.

\begin{ter}\label{ter: stack}
Let $\base$ be a base scheme. Throughout this paper, a \tdef{stack}
over $\base$ is a sheaf of (1-truncated) spaces on the big \'etale site of $\base$ which is 
\emph{nicely scalloped} in the sense of~\cite[\href{https://arxiv.org/pdf/2106.15001\#equation.2.9}{Def.~2.9}]{kr24}
unless otherwise specified. 
\end{ter}

Having declared this once and for all, we also fix the following pieces of notation.

\begin{notation}
For a base scheme $\base$, we write $\mdef{\Sch_{\base}}$ for the category of qcqs
schemes over $\base$. Moreover, we denote by $\mdef{\Stk{\base}}$ the category of 
stacks over $\base$ and given any $\stack \in \Stk{\base}$, we often leave the base implicit
and write $\Stk{\stack}= (\Stk{\base})_{/\stack}$ for the slice category.
Moreover, we write $\mdef{\Sm{\stack}} \subseteq \Stk{\stack}$ for the full subcategory
spanned by the smooth, representable morphisms of finite presentation.
\end{notation}

Beware that, the notion of stack of \cref{ter: stack} is relative to the base $\base$. 
Its primary function is to \emph{exclude} stacks that are badly behaved, such
as the classifying stack of a finite group whose order is divisible by a residue field.
However, the only stacks that are essential for us are the following, which are well behaved
over any base.

\begin{example}
For any scheme $\base$, a \tdef{torus} over $\base$ is an abelian group scheme $\t$,
with the property that after a finite \'etale extension $\base^{\prime}\to \base$
it splits as $\t_{\base^\prime} \simeq \Gm^{\times n}$ for some $n\geq 1$. 
For any torus $\t/\base$, the \'etale classifying stack 
$\deloop{\t}$ is nicely scalloped by~\cite[\href{https://arxiv.org/pdf/2106.15001\#equation.2.2}{Ex.~2.2}]{kr24}. Moreover, for any injective group homomorphism of tori 
$\wt \to \t$ the \'etale delooping  $\deloop{\wt}\to \deloop{\t}$ 
is smooth, representable and of finite presentation and so defines an object of 
$\Sm{\deloop{\t}}$. 
\end{example}

We now outline the construction of the stable motivic category $\SH{\stack}$ of a stack $\stack$, in this form due to~\cite{kr24}. 

Fix a base scheme $\base$ and for any $\fX\in \Stk{\base}$ write $\mdef{\Nistop{\fX}}$
for the category of Nisnevich sheaves of spaces on $\Sm{\fX}$. 
The category of \tdef{motivic spaces}
is defined to be the full subcategory $\mdef{\rm{H}(\stack)} \subseteq\Nistop{\stack}$
spanned by the $\AA^1$-invariant sheaves. The fully faithful inclusion
$\rm{H}(\stack)\hookrightarrow \Nistop{\fX}$ admits a product preserving left adjoint
which we denote by
\[ \mdef{L_{\AA^1}}\colon \Nistop{\base}\too \rm{H}(\stack)\,,\] see \cite[\href{https://arxiv.org/pdf/2106.15001\#equation.3.4}{Rem.~3.4}]{kr24}.
If $\sE$ is a vector bundle on $\stack$ with total space $\sV(\sE)$, we include
$\stack \hookrightarrow \sV(\sE)$ via the zero section and denote by
\[ \mdef{\rm{Th}(\sE)} \coloneq \cofib\left( \sV(\sE)\setminus \stack \to \sV(\sE)\right)
\in \rm{H}(\stack)_{\ast}\]
the cofiber in pointed motivic spaces, which we call the \tdef{motivic Thom space} of $\sV$.

If the stack $\stack$ is \emph{nicely quasi-fundamental} in the sense 
of~\cite[\href{https://arxiv.org/pdf/2106.15001\#equation.2.9}{Def.~2.9}]{kr24}
the category of \tdef{motivic spectra} over $\stack$ is defined by formally $\otimes$-inverting all
motivic Thom spaces in $\rm{H}(\stack)_{\ast}$ with respect to the smash product, i.e.~we have 
\[ \mdef{\SH{\stack}}\coloneq \rm{H}(\stack)_{\ast}[\{\rm{Th}(\sE)^{-1}\}_{\sE}]\]
where $\sE$ runs over all vector bundles on $\stack$. The general case construction is
obtained via descent and we refer the reader to the proof of~\cite[\href{https://arxiv.org/pdf/2106.15001\#equation.4.5}{Thm.~4.5}]{kr24} for details.

The category $\SH{\stack}$ is stable, naturally refines to an object of $\CAlg(\Prl)$ and
we denote the monoidal unit by $\SS_{\fX}\in \SH{\stack}$, the \tdef{motivic sphere spectrum}.
By construction, $\SH{\stack}$ comes equipped with a suspension spectrum functor 
$\Sigma^\infty \colon \rm{H}(\stack)_{\pt} \to \SH{\stack}\,.$
We often implicitly consider smooth representable stacks $\fZ\in \Sm{\stack}$ as
Nisnevich sheaves of spaces on $\Sm{\fX}$ via the Yoneda-embedding and
denote the composition of $\AA^1$-localization and (unpointed) motivic suspension by
\[ \mdef{\SS_{\stack}[-]}\colon \Nistop{\stack} \too \SH{\stack}.\]
This functor is symmetric monoidal, colimit preserving and 
left Kan extended from the full subcategory $\Sm{\stack}\subseteq \Nistop{\stack}$.

If $f\colon \fZ\to \fX$ is any map of stacks, pullback of sheaves induces a symmetric
monoidal left adjoint functor
\[ \mdef{f^{\ast}} \colon \SH{\fX}\too \SH{\fZ}\]
whose lax symmetric 
monoidal right adjoint we denote by
\[ \mdef{f_{\ast}}\colon \SH{\fZ}\too \SH{\fX}.\]
Moreover, by \cite[\href{https://arxiv.org/pdf/2106.15001\#equation.4.5}{Thm.~4.5}]{kr24} the assignment taking a stack to its category of motivic spectra and taking a map to its $(-)^{\ast}$ refines to a functor
\[ \pSH^{\ast}\colon \Stk{\base}^\op \too \CAlg(\Prl) \]
which satisfies Nisnevich descent.

\begin{thm}[{\cite[\href{https://arxiv.org/pdf/2106.15001\#equation.4.10}{Thm.~4.10}]{kr24}}] \label{thm: lower sharp proj formula}
   Let $f\colon \fZ \to \fX$ be smooth and representable. Then the 
   functor $f^{\ast}$ admits a left adjoint 
   \[ \mdef{f_{\sharp}} \colon \SH{\fZ} \too \SH{\fX},\]
   which fulfills base change along $(-)^{\ast}$ and is $\SH{\fX}$-linear, meaning for
   all $A\in \SH{\fX}$ and $B\in \SH{\fZ}$, we have a projection formula
   \[ f_{\sharp}(f^{\ast}A \otimes B) \simeq A\otimes f_{\sharp} B \in \SH{\fX}.\]
   Moreover, $f_{\sharp}$ commutes with suspension, so in particular we have
    \[ f_{\sharp} \SS_{\fZ} \simeq \SS_{\fX}[\fZ] \in \SH{\fX}.\]
\end{thm}

By general internal adjunction yoga, the smooth projection formula implies
that, for all $A\in \SH{\fZ}$ and $B\in \SH{\fX}$, we have a natural equivalence
\[ \iMap(f_{\sharp}A,B) \simeq f_{\ast}\iMap(A,f^{\ast}B).\]
In particular, taking $B=\SS_{\fX}$ we learn that we have a weak duality
$(f_{\sharp}(-))^{\vee}\simeq f_{\ast}\left((-)^{\vee}\right)$.

\begin{thm}\label{thm: proper proj formula}
   Let $f\colon \fZ\to \fX$ be representable, finite type and proper. Then, the functor 
    \[f_{\ast}\colon \SH{\fZ}\too \SH{\fX}\]
   commutes with colimits and is $\SH{\fX}$-linear, meaning for all 
   $A\in \SH{\fX}$ and $B\in \SH{\fZ}$ we have a projection formula
   \[ A \otimes f_{\ast}B \simeq f_{\ast}(f^{\ast}A\otimes B). \]
\end{thm}
\begin{proof}
    This follows by combining statements (ii) and (iv) in~\cite[\href{https://arxiv.org/pdf/2106.15001\#equation.7.1}{Thm.~7.1}]{kr24}.
\end{proof}

\begin{thm}[{\cite[\href{https://arxiv.org/pdf/2106.15001\#equation.6.1}{Thm.~6.1.~(ii)~\&~(iii)}]{kr24}}]\label{thm: base change formulas}
Let $f\colon \fZ\to \fX$ be a representable map of stacks and suppose
we have a pullback diagram of stacks of the form:
    \[\begin{tikzcd}
	\mathfrak{P} & \fY \\
	\fZ & \fX
	\arrow["g", from=1-1, to=1-2]
	\arrow["q"', from=1-1, to=2-1]
	\arrow["p", from=1-2, to=2-2]
	\arrow["f"', from=2-1, to=2-2]
\end{tikzcd}\]
If $f$ is proper, the Beck-Chevalley transform
\[p^{\ast}f_{\ast}\too g_{\ast}q^{\ast} \]
is an equivalence. Moreover, if $p$ is additionally smooth and representable,
then we also have a natural equivalence
\[ p_{\sharp}g_{\ast} \xrightarrow{\sim} f_{\sharp}q_{\ast}.\]
\end{thm}

Let $\sE$ be a vector bundle on $\fX$. We denote the motivic suspension spectrum of 
the Thom-space of $\sE$ by 
\[ \mdef{\SS_{\fX}\twist{\sE}}\coloneq \Sigma^{\infty}\mathrm{Th}(\sE)\in \SH{\fX}\,,\]
and refer to $\SS_{\fX}\twist{\sE}$ as the \tdef{motivic Thom spectrum} of $\sE$.
 By construction of $\SH{\fX}$, all motivic Thom spectra $\SS_{\fX}\twist \sE$
 are $\otimes$-invertible and hence this construction refines to a functor 
\[ \SS_{\fX}\twist{-}\colon K_0(\fX) \too \mathrm{Pic}(\SH{\fX})\,.\]
We then have the following important duality statement:

\begin{thm}[Atiyah Duality]\label{thm: Atiyah Duality}
   Let $f\colon \fZ\to \fX$ be a smooth and representable map of stacks and denote 
   by $\mathcal L_{f}$ the cotangent complex of $f$. 
   \begin{enumerate}
       \item There exists a natural transformation
   \[ f_{\sharp}\twist{-\mathcal L_{f}} \too f_{\ast},\]
   which is an equivalence if $f$ is proper. 
        \item Let $A\in \SH{\fZ}$ be a dualizable object. If $f$ is proper, then $f_{\sharp}A$ is dualizable as well with dual given by 
   $f_{\ast} (A^{\vee})\simeq f_{\sharp}(A^{\vee}\twist{-\mathcal L_f})$.
   \end{enumerate}
\end{thm}
\begin{proof}
    The first claim is~\cite[\href{https://arxiv.org/pdf/2106.15001\#equation.6.7}{Cons.~6.7.}]{kr24} combined with~\cite[\href{https://arxiv.org/pdf/2106.15001\#equation.6.1}{Thm.~6.1.}(iv)]{kr24}.
    For the second item we use the smooth and proper projection formulas of 
    \cref{thm: lower sharp proj formula} and \cref{thm: proper proj formula} to see
    that, for such $A\in \SH{\fZ}$ and all $B,C\in \SH{\fX}$ we have natural equivalences
   \begin{align*}
       \Map(B\otimes f_{\sharp}A, C)&\simeq \Map(f_{\sharp}(f^{\ast} B \otimes A), C)\\
       &\simeq \Map(f^{\ast}B \otimes A, f^{\ast}C)\\
       &\simeq \Map(f^{\ast}B, A^{\vee}\otimes f^{\ast}C)\\
       &\simeq \Map(B, f_{\ast}(A^{\vee} \otimes f^{\ast}C)\\
       &\simeq \Map(B, f_{\ast}A^{\vee} \otimes C)
   \end{align*} 
    so the claim follows.
\end{proof}

\subsection{K-theory and motivic Bott periodicity}\label{ssec: mot k theory}

Let $\Catperf$ be the category of small, stable, idempotent complete categories
with exact functors between them.
We denote the lax symmetric monoidal Bass--Thomason--Trobaugh $K$-theory construction
by
\[ \mdef{K}\colon \Catperf \too \Sp\,,\]
which we simply refer to as \tdef{algebraic $K$-theory}.

The assignment taking a scheme $X$ to its 
(derived) category of quasi-coherent sheaves $\QCoh(X)$ refines to a functor 
from $\mathrm{Sch}^{\op}$ to the category of stable, presentably symmetric monoidal categories 
$\CAlg(\Prlst)$. Moreover, $\QCoh(-)$ satisfies \'etale descent, so we may Kan extend
it to all \'etale sheaves of spaces and obtain a functor
\[ \QCoh(-)\colon \Shv_{\et}(\mathrm{Sch}; \Spc)^{\op} \too \CAlg(\Prlst) \,.\]
For any arbitrary $A\in \Shv_{\et}(\mathrm{Sch})$ the category $\QCoh(A)$ thus defined
has no reason to be compactly generated. However, we know, for example by~\cite[\href{https://arxiv.org/pdf/2106.15001\#equation.2.24}{Thm.~2.24}]{kr24},
that any $\stack \in \Stk{\base}$ is \emph{perfect} i.e.~the functor $\QCoh(-)$ 
factors through the (non-full) subcategory $\cPrlst \subseteq \Prlst$ of compactly
generated categories with compact object preserving functors between them.
Hence, for any stack $\fX\in \Stk{\base}$, we define the category of 
perfect complexes on $\fX$ as the category of compact objects 
\[\mdef{\Perf(\fX)}\coloneq \QCoh(X)^{\omega} \in \CAlg(\Catperf)\,.\]
This lets us apply the algebraic $K$-theory functor and we write
\[ \mdef{K(\fX)}\coloneq K(\Perf(\fX))\in \CAlg(\Sp)\]
which we simply call the \tdef{$K$-theory} spectrum of $\fX$.

By~\cite[\href{https://arxiv.org/pdf/2106.15001\#equation.10.2}{Thm.~10.2}]{kr24}, 
generalizing the original result of Thomason and Trobaugh~\cite{tt90},
the composite functor 
\[ \Stk{\stack}^{\op} \too \Sp, \quad \fZ \mapsto K(\fZ)\]
satisfies Nisnevich descent. Restricting to the smooth locus
$\Sm{\stack}$ we may take its $\AA^1$-localization to obtain an $S^1$-spectrum
in motivic spaces
\[ \mdef{\mathrm{KH}_{\stack}}\coloneq L_{\AA^1}K(-) \in \CAlg(\Nistop{\stack} \otimes \Sp)\]
which we refer to as \tdef{homotopy $K$-theory}. As discussed 
in~\cite[\href{https://arxiv.org/pdf/2106.15001\#equation.10.3}{Cons.~10.3}]{kr24},
if $\fX$ is regular, then we have 
$\mathrm{KH}_{\fX}(\fX)\simeq K(\fX)=K(\QCoh(\fX))$. 
Moreover, by \cite[\href{https://arxiv.org/pdf/2106.15001\#equation.10.4}{Rmk.~10.4}]{kr24},
we have for any representable map of stacks $f\colon \fZ\to \fX$ a natural equivalence 
\[ f^{\ast}\KH_{\fX} \simeq \KH_{\fZ}\,.\]

\begin{defn}\label{defn: KGL}
By~\cite[\href{https://arxiv.org/pdf/2106.15001\#equation.10.7}{Thm~10.7}]{kr24}
the sheaf of spectra $\mathrm{KH}_{\stack}$ is representable by a motivic 
$\EE_\infty$-ring spectrum
\[\mdef{\KGL_{\stack}}\in \CAlg(\SH{\stack})\, ,\]
such that, for any smooth and representable map $f\colon \fZ \to \stack \in \Sm{\stack}$, 
we have a natural equivalence $f^\ast \KGL_{\stack} \simeq \KGL_{\fZ}$.
We call $\KGL_{\stack}$ the \tdef{motivic $K$-theory spectrum} over $\stack$,
\end{defn}

\begin{notation}\label{not: KGL cohom}
For any map of stacks $f\colon \fZ\to \fX$, we write
\[ \mdef{\KGL_{\fX}^{\fZ}}\coloneq f_{\ast} \KGL_{\fZ}\in \CAlg_{\KGL_{\fX}}(\SH{\fX})\]
and refer to this motivic spectrum as the $\KGL_{\fX}$-cohomology of $\fZ$.
\end{notation}

\begin{war}\label{war: completed KGL}
Unlike for the motivic sphere, if $f\colon \fZ\to \fX$ is not representable, 
the motive $\KGL_{\fX}^{\fZ}$ generally does not agree with 
the naive cohomology $f_{\ast}f^{\ast}\KGL_{\fX}$. Indeed,
unwinding the definitions and using the smooth base change formula of~\cite[\href{https://arxiv.org/pdf/2106.15001\#equation.6.1}{Thm.~4.10.(i).(b)}]{kr24},
we learn that $\KGL_{\fX}^{\fZ}$ has underlying sheaf of spectra on $\Sm{\fX}$ 
given by
\[ \ulPsh{\fX}(\KGL_{\fX}^{\fZ})(\fA) = \KH(\fY \times_{\fX} \fA)\,\]
Meanwhile, $f_*f^*\KGL_\fX$ is simply the right Kan extension of its restriction
to smooth representable stacks over $\fX$ and will act as a suitable completed version of
$\KGL_{\fX}^{\fZ}$, as explained in \cite{elmanto2025filteredalgebraicktheorystacks}.
\end{war}

The motivic spectrum $\KGL_{\fX}$ satisfies a much stronger version of Atiyah duality,
namely \tdef{Bott-periodicity}. Namely, by \cite[\href{https://arxiv.org/pdf/2106.15001\#equation.10.7}{Thm.~10.7}]{kr24}, for any virtual vector bundle 
$\mathcal E \in K_0(\fX)$, there is a natural equivalence
\[ \KGL_\fX\twist{\mathcal \sE} \cong \KGL_\fX \,. \]
This allows us to describe the category $\KGL_{\fX}$-modules as a localization
of the category $\KH_{\fX}$-modules, which will be useful for us in 
\cref{ssec: motivic Fourier}. 
Denote by $\Sp(\MSpc{\fX})$ the category of $S^1$-spectra in motivic spaces,~i.e. spectral
$\AA^1$-invariant Nisnevich sheaves on $\Sm{\fX}$. 
The description in this form is due to Hoyois~\cite{Hoyois_2020}.

\begin{prop} \label{prop:Bott-periodization-is-localization}
    For any stack $\fX$, the forgetful functor
    \[ \Mod{\SH{\fX}}{\KGL_{\fX}} \too \Mod{\Sp(\MSpc{\fX}_{\ast}}{\KH_{\fX}}\]
    is fully faithful. 
\end{prop}
\begin{proof}
    Let us first consider the case where $\stack$ is quasi-fundamental in the sense of \cite[\href{https://arxiv.org/pdf/2106.15001\#equation.2.9}{Def.~2.9}]{kr24}, in which case
    the stable motivic category $\SH{\stack}$ is obtained from $\Sp(\MSpc{\fX}_{\ast})$ by
    $\otimes$-inverting Thom spaces.  

    In this case, we even have a description of the essential image: It is spanned by the Bott-periodic modules. That is, as explained in \cite[Sec.~4]{Hoyois_2020}, for each vector bundle $\sE$ on $\fX$ we may use the projective bundle 
    formula to construct
    a map of $S^1$-spectra
    \[ \beta_{\sE}\colon \Sigma^{\infty}_{S^1}\mathrm{Th}(\sE) \too \KH_{\fX}\]
    called the Bott element of $\sE$.
    A $\KH_\stack$-module is called Bott-periodic if it is local with respect to those morphisms. 
    The fully faithfulness claim then follows from~\cite[Prop.~3.2]{Hoyois_2020}, using the fact that Thom-spectra are invertible
    in $\SH{\fX}$ and so Bott-periodicity for $\KGL_{\fX}$-modules is automatic. 

    For the general case, we can choose a Nisnevich cover of any stack by nicely quasi-fundamental stacks. As the functor $u_{\fX}$ commutes with smooth $(-)^\ast$ (as $(-)_\sharp$ commutes with the adjoint), we can deduce the general statement by Nisnevich descent, using that fully faithful functors are closed under limits.  
\end{proof}
In particular, we may identify the category $\Mod{\SH{\fX}}{\KGL}$ with a Bousfield localization.
We denote the symmetric monoidal left adjoint by
\[\mdef{L_{\beta}}\colon \Mod{\Sp(\MSpc{\fX})}{\KH_{\fX}}\too \Mod{\SH{\fX}}{\KGL_{\fX}}\]
and refer to it as \tdef{Bott-periodization}. 

\subsection{Equivariant and motivic homotopy theory}\label{ssec: motivic-vs-equivariant}

In this section, we import some equivariant technology, as well
as a comparison to motivic homotopy theory, of which we give a brief
summary. We direct the reader towards~\cite[Sec.~10]{bh21} for an in-depth
discussion.

Let $\baf$ be a field, choose a separable closure $\baf\to \bar{\baf}$ 
and denote the absolute Galois group by~$\GalG$.
As $G$ is a pro-finite group, the theory of $\GalG$-equivariant homotopy
theory is slightly modified from the case of a finite group.

\begin{defn}
For $G$ a pro-finite group, we write $\Fin_G$ for the category of finite continuous $G$-sets and $\mathcal O_G$ for the subcategory of transitive $G$-sets. We write
    \[ \mdef{\Spc^{BG}} = \lim_{N<G \, \text{finite index}} \Spc^{BG/N} \,, \] where the limit is taking along taking homotopy fixed points, for the \mdef{category of Borel $G$-spaces} and \[ \mdef{\Spc_G} = \Fun(\mathcal O_G^{\op},\Spc) \cong \lim_{N<G \, \text{finite index}} \Spc_{G/N} \] where the limit is taken along taking fixed points for the category of \mdef{genuine $G$-spaces}.
The inclusion of Borel-equivariant spaces into genuine spaces, i.e.~the right adjoint to restriction, is denoted by
\[ \mdef{\beta} \colon \Spc^{BG} \too \Spc_G \,. \]
Finally, we write \mdef{$\Sp_{\GalG}$} for 
the category of \tdef{genuine $\GalG$-spectra} as defined in \cite[Sec.~9]{bh21}.
\end{defn}

This category of genuine $G$-spectra was first constructed by Barwick~\cite{bar17} and
his construction agrees with the one given by Bachmann and Hoyois by~\cite[Prop.~9.11]{bh21}.

We regard $\Fin_G$ as a full
subcategory of $\Sm{\baf}$ via the unique coproduct preserving functor
determined by the assignment
\begin{equation}\label{eq: G-sets into schemes}
 i\colon \Fin_{\GalG}\too \Sm{\baf}, \quad \GalG/H \mapsto \bar{\baf}^{H}.
\end{equation}
The subcategory $\Fin_G \subseteq \Sm{\baf}$ is closed under taking both Nisnevich and \'etale
covers, so that $\Fin_G$ defines a sub-site of $\Sm{\baf}$ in either topology. 
Let $\tau\in \{\Nis, \et\}$, then restriction and left Kan extension along $i$
induces an adjunction
\[c_{\tau} \colon \Shv_{\tau}(\Fin_G) \adj \Shv_{\tau}(\Sm{F}) 
\noloc u_{\tau}.\] 
Recall that we had an adjunction 
\[ \nu_{\ast} \colon \Shv_{\et}(\Sm{\baf}) \adj \Shv_{\Nis}(\Sm{\baf})\noloc \nu^{\ast} \]
where $\nu_{\ast}$ is the fully faithful inclusion and $\nu^{\ast}$ is
sheafification.
Now let $\cO_G\subseteq \Fin_G$ denote the full subcategory spanned by
the transitive $G$-sets and consider the following diagram of sheaf categories:
\[ \begin{tikzcd} \Shv_{\et}(\Sm{F}) \ar[d,"{u_{\et}}"] \ar[r,"\nu_{\ast}"] & \Shv_{\Nis}(\Sm{F}) \ar[d,"{u_{\Nis}}"] \\
 \Shv_{\et}(\Fin_G) \ar[r,"\nu_{\ast}"] \ar[d,"\simeq"] & \Shv_{\Nis}(\Fin_G) \ar[d,"\simeq"] \\
 \Spc^{BG} \ar[r,"\beta"] & \Spc_G
 \end{tikzcd} \] 
 Here, in the upper square the horizontal maps are the canonical inclusions and the vertical maps are
 given by restriction along the maps of the respective sites. 
 The bottom right vertical equivalence is given by restriction along the transitive,
finite $G$-sets $\mathcal O_G$. It restricts to the equivalence on the left between Borel $G$-spaces and étale sheaves.

Since everything in sight is (lax) symmetric monoidal, applying $\CAlg(-)$ and restricting 
to group-like objects, we get a commutative diagram of right adjoints: 
\[\begin{tikzcd}
	{\etSpcn(F)} & {\etSpcn(F)} \\
	{\Spcn({\deloop{\GalG}}}) & {\Spcn({\cO_{\GalG}^{\op}}})\, ,
	\arrow["{\nu_{\ast}}", from=1-1, to=1-2]
	\arrow["{U_{\et}}"', from=1-1, to=2-1]
	\arrow["{U_{\Nis}}", from=1-2, to=2-2]
	\arrow["\beta"', from=2-1, to=2-2]
\end{tikzcd}\]
such that $u_{\tau} \ulSpc{}= \ulSpc{U_{\tau}}$ where $\tau$ denotes either topology. Accordingly,
we obtain adjunctions on sheaves of spectra which we denote
\[ C_{\et} \colon \Spcn(\deloop{\GalG}) \rightleftarrows  \Spcn_{\et}(\baf) \noloc U_{\et}\]
and 
\[ C_{\Nis} \colon \Spcn( \cO_{\GalG}^{\op}) \rightleftarrows \Spcn_{\Nis}(\baf) \noloc U_{\Nis}\,.\]

We also record the following:
\begin{lem} \label{lem:C-fully-faithful}
    For $\tau \in \{ \et,\Nis \}$ the left adjoints $c_\tau, C_\tau$ are fully faithful.
    Moreover, the right adjoints $u_\tau, U_\tau$ preserve colimits.
\end{lem}
\begin{proof}
    We claim that the unit of the adjunction  $c_\tau \dashv u_\tau$ is an equivalence.
     This is true for the adjunction on presheaf categories, as $\Fin_G \subset \Sm{F}$ is 
    fully faithful. Moreover, as $\Fin_G$ is closed under all covers, sheafification does
    not change the value of the left adjoint at elements in $\Fin_G$. The argument
    for fully faithfulness of $C_\tau$ is exactly the same. 
    
    Similarly, we observe that the right adjoints preserve colimits on the level
    of presheaf categories. As colimits are computed by taking the colimit 
    pointwise and then sheafifying, the claim follows by the same argument
    as the fully faithfulness.
\end{proof}

Lastly, these adjunctions actually refine to the level of $\pSH$, which we 
record as the following proposition.

\begin{prop} \label{prop: c-commutes-with-suspension}
Restriction and Kan-extension along the inclusion $i\colon \rm{Fin}_{\GalG} \to \Sm{\baf}$
induces an adjunction
\[ \cC\colon \Sp_{\GalG}\adj \SH{\baf} \noloc \cU,\]
such that $\cC$ is symmetric monoidal and makes the following diagram commute:
    \[ \begin{tikzcd} 
        \Shv_{\Nis}(\Fin_\GalG) \ar[d,"{c_{\Nis}}"] \ar[r,"{\mathbb{S}[-]}"] & \GSp{\GalG} \ar[d,"\cC"] \\
        \Shv_{\Nis}(\Sm{\baf}) \ar[r,"{\mathbb{S}[-]}"] & \SH{\baf}\, .
        \end{tikzcd} \] 
\end{prop}
\begin{proof}
    This is precisely~\cite[\href{https://arxiv.org/pdf/1711.03061\#equation.10.6}{Prop.~10.6}]{bh21}, noting that, since everything in sight
    commutes with colimits, it suffices to check commutativity on representables 
    $\GalG/H \in \Fin_{\GalG}$. 
\end{proof}

The right adjoint $\cU$ now gives a convenient definition of the 
$\GalG$-equivariant algebraic $K$-theory spectrum.

\begin{defn}\label{defn: G equiv K theory}
    For any $X\in \Sm{\baf}$, we define the \tdef{$\GalG$-equivariant $K$-theory}
    of $X$ as the genuine $\GalG$-spectrum 
    \[\mdef{K_{\GalG}(X)}\coloneq \mathcal{U} (\KGL_{\baf}^{X})
    \in \CAlg(\GSp{\GalG}),\]
    where $\cU\colon \SH{\baf}\to \GSp{\GalG}$ is the right adjoint
    of \cref{prop: c-commutes-with-suspension}.
\end{defn}

\begin{rem}
Explicitly, we have that, for any open subgroup $H\subseteq \GalG$, the 
$\GalG$-spectrum $K^{\GalG}(X)$ has $H$-fixed points given by the formula
 \[K_\GalG(X)^H= K(X_{\bar{\baf}^H}).\]
Note that, since $K$-theory does in general not satisfy \'etale descent, this 
$\GalG$-spectrum is typically not Borel.
\end{rem}

\subsection{Recollections on parametrized category theory}\label{ssec: parametrized}

In this subsection we fix a topos $\tB$ throughout. We will remind the reader of some constructions within parametrized category theory, that is,
the category theory internal to $\tB$, as developed by Martini--Wolf in~\cite{mwyoneda},~\cite{mwcolimit} and~\cite{mwpresentable}.

Among other things, we recall the notions of parametrized Kan extensions, 
presentability and the Lurie tensor product. Moreover, we discuss some aspects of 
higher algebra in the $\tB$-parametrized setting.

\begin{defn}
Let $\tB$ be a topos. A \tdef{$\tB$-parametrized category} $\cC$ is a 
limit preserving functor
    \[\cC \colon \tB^\op\longrightarrow \BigCat \,, \quad U \mapsto \cC(U)\,. \]
    The collection of $\tB$-categories forms a large category, which we denote by 
    $\mdef{\BigCat(\tB)}\coloneq \Fun^{\lim}(\tB^\op,\BigCat)$. A $\tB$-category is called small if it factors through $\Cat \subset \BigCat$.
\end{defn}

Unless otherwise specified, for a $\tB$-category $\cC$ and a map 
$f\colon U\to V$ in $\tB$, we denote by $f^{\ast}\colon \cC(V)\to \cC(U)$ 
the \enquote{pullback} functor supplied by the data of $\cC$. As usual, we also write
$f_{\sharp} \dashv f^{\ast} \dashv f_{\ast}$ for the adjoints, if they exist.

Evaluating at any object $U\in \tB$ defines a sections functor which we denote by
\[ \Gamma_{U}\colon \BigCat(\tB)\too \BigCat, \quad \cC \mapsto \Gamma_U(\cC)= \cC(U).\]
For $U= \pt$ being the terminal object of $\tB$, we refer to
$\Gamma(\cC)\coloneq \Gamma_{\pt}(\cC)$ as the global sections of $\cC$. 
For $U\in \tB$, an \tdef{element of $\cC$ in context $U$} is an element of $\cC(U)$.
If no context is specified, an element of $\cC$ is understand to mean an element of
the global sections $\Gamma(\cC)$.

\begin{example}
The following are the prototypical examples of $\tB$-categories.
    \begin{enumerate}
    \item  For any topos $\tB$, the slice functor
    \[\mdef{\Omega_\tB}:\tB^{\op} \too \BigCat \quad b\mapsto\tB_{/b},\] 
     defines a $\tB$-category that we refer to as the \tdef{universe} of $\tB$.
    \item  For any $A\in \tB$, the mapping space functor
        \[ \Map_{\tB}(-,A)\colon \tB^{\op}\too \Spc \subseteq \Cat\]
        defines a small $\tB$-category. In fact, this construction yields an equivalence
        $\Fun^{\rm{lim}}(\tB^{\op}, \Spc)\simeq \tB$ and so we freely consider
        objects of $\tB$ as $\tB$-categories.
    \item Let $\tB=\Shv(\cal{X}, \Spc)$ be a sheaf topos on some site $\cal{X}$.
          Then, right Kan extension along the fully faithful inclusion $\cal{X}\to \tB^{\op}$  
          defines an equivalence of categories
          \[ \Shv(\cal{X},\BigCat) \xrightarrow{\sim} \BigCat(\tB).\]
    \end{enumerate}
  \end{example}

  \begin{defn}\label{defn: param functor cat}
    By~\cite[\href{https://arxiv.org/pdf/2111.14495\#equation.2.6.4}{Prop.~2.6.4}]{mwcolimit} the category $\Cat(\tB)$ is cartesian closed.
    Thus, for any $\cC,\cD\in \BigCat(\tB)$ we have a $\tB$-category $[\cC,\cD]$ 
    of $\tB$-parametrized functors from $\cC$ to $\cD$. We write
    \[\mdef{\Fun_\tB(\cC,\cD)}\coloneq \Gamma_*[\cC,\cD]\in\BigCat\] 
    for the category of global sections.
  \end{defn}

  The content of \cref{defn: param functor cat} provides a natural refinement of $\BigCat(\tB)$
  to an $(\infty,2)$-category, and hence one can talk about adjunctions internal to $\BigCat(\tB)$
  as is done in~\cite[\href{https://arxiv.org/pdf/2111.14495\#equation.3.1.1}{Def. 3.1.1}]{mwcolimit}. 
  We will not discuss these notions in detail and instead skip ahead to parametrized
  Kan extensions and then work backwards. Note that, for any $\tB$-functor $f\colon \cC\to \cD$ 
  we have an induced natural transformation $f^{\ast}\colon [\cD,-]\to [\cC,-]$
  of functors $\BigCat(\tB)\to \BigCat(\tB)$. The following definition is precisely\footnote{Beware that, we denote left adjoints to pullback generically by $(-)_{\sharp}$, while Martini--Wolf use $(-)_{!}$.} \cite[\href{https://arxiv.org/pdf/2111.14495\#equation.6.3.1}{Def. 6.3.1}]{mwcolimit} and will be important for us.
  
  \begin{defn} \label{defn: kan extension}
   Let $f\colon \cC \to \cD$ and $g\colon \cC\to \cE$ be $\tB$-functors.
    A \tdef{left Kan extension} of $g$ along $f$ is a $\tB$-functor 
    \[\mdef{f_{\sharp}g}:\cD\rightarrow \cE,\] 
    together with an equivalence of $\tB$-functors $[\cD,\cE]\to \Omega_{\tB}$
    of the form
    \[\Map_{[\cD,\cE]}(f_{\sharp}g,-) \simeq \Map_{[\cC,\cD]}(g,f^{\ast}(-))\,. \]
    Dually, a \tdef{right Kan extension} along $f$ is a functor $\mdef{f_{\ast}g}:\cD\to \cE$,
    together with an equivalence
    \[\Map_{[\cD,\cE]}(-,f_*g)\simeq\Map_{[\cC,\cE]}(f^*(-),g).\]
    If $f$ is the unique functor to the terminal $\tB$-category $\cC\to \pt$ 
    we call $f_{\sharp}g$ the \tdef{parametrized colimit} and $f_{\ast}g$
    the \tdef{parametrized limit} of $g$.
  \end{defn}

\begin{example}\label{example: parametrized colimit}
  Let $\cC$ be a $\tB$-category and let $A\in \tB$ with terminal map
  $p\colon A\to \pt$. Let $X$ be an object of $\cC(A)$, considered as a $\tB$-functor
  $A \xrightarrow{E} \cC$. Unwinding the definitions, we see
  that, if this functor admits a parametrized colimit, there exists an
  object $p_{\sharp}E \in \cC(\ast)$, together with an equivalence 
  \[ \eta\colon \Map_{\cC(\pt)}(p_{\sharp}E,-) \simeq \Map_{\cC(A)}(E, p^{\ast}(-))\,. \]
  In particular, if $\cC$ admits all parametrized colimits of this form,
  the functor $p^{\ast}$ necessarily admits a left adjoint $p_{\sharp}$. Dually,
  one gets that the parametrized limit is computed via a faux right adjoint~$p_{\ast}$.
  We also write $E^{A}\coloneq p_{\ast}E \in \cC$ and 
  $E[A] \coloneq p_{\sharp}E$, whenever they exist.
\end{example}

\begin{defn}
    A $\tB$-category $\cD$ is called (co)complete if any functor $f \colon \cC \to \cD$ with $\cC$ small, admits a $\tB$-(co)limit.

 Moreover, $\cD$ is called \tdef{presentable} if it is cocomplete and sectionwise presentable.
 We write $\mdef{\pPrl{\tB}}$ for the large category of presentable $\tB$-categories and colimit-preserving functors.
\end{defn}

\begin{rem}\label{presentable defn}
It is shown in \cite[\href{https://arxiv.org/pdf/2111.14495\#equation.5.4.7}{Cor.~5.4.7}]{mwcolimit} that a $\tB$-category is cocomplete if and only if 
  \begin{enumerate}
        \item  it factors through the large category $\BigCat^{\rm{R}}_{\rm{cc}}$ of 
        cocomplete categories with colimit preserving right adjoint functors between them.
        \item  The adjoints $(-)_{\sharp} \dashv (-)^{\ast}$ satisfy the Beck-Chevalley
        condition for any pullback square in $\tB$.
    \end{enumerate} 
\end{rem}

\begin{thm}[{\cite[\href{https://arxiv.org/pdf/2209.05103v2\#equation.2.6.3.2}{Prop.~2.6.3.2} \& \href{https://arxiv.org/pdf/2209.05103v2\#equation.2.6.3.7}{Prop~2.6.3.7}]{mwpresentable}}]\label{colim completion}  
 The large category of presentable $\tB$-categories $\pPrl{\tB}$ admits a symmetric monoidal
  structure with unit being the universe $\Omega_\tB$. Moreover, the global sections
  functor 
  \[\Gamma(-)\colon \pPrl{\tB}\too \Mod{\Prl}{\tB}\]
  admits a strong symmetric monoidal, fully faithful and colimit preserving left adjoint
  \[ -\otimes_{\tB} \Omega_{\tB}\colon \Mod{\Prl}{\tB} \too \pPrl{\tB}.\]
\end{thm}

\begin{defn}\label{defn: param sheaves}
Equipped with the symmetric monoidal structure of \cref{colim completion}, we call 
$\mdef{\CAlg(\pPrl{\tB}})$ the large category of \tdef{presentably symmetric monoidal} $\tB$-categories. 
We also denote the composition
\[\Prl \xrightarrow{-\otimes\tB} \Mod{\Prl}{\tB}
  \xrightarrow{-\otimes_{\tB}\Omega_{\tB}} \pPrl{\tB}\] 
simply by $-\otimes \Omega_{\tB}$. For any
presentable category $\cC\in \Prl$ we call 
\[\mdef{\cC \otimes \Omega_{\tB}} \in \pPrl{\tB}\] 
the presentable $\tB$-category of \tdef{$\cC$-valued sheaves on $\tB$}.
\end{defn}

Since the functor $-\otimes \Omega_{\tB}$ is symmetric monoidal by \cref{colim completion}, it 
carries presentably symmetric monoidal categories to presentably symmetric monoidal $\tB$-categories.

\begin{rem}\label{rem: param modules + algebras}
  In~\cite[\href{https://arxiv.org/pdf/2209.05103v2\#section.2.5}{Sec.~2.5}]{mwpresentable}, Martini--Wolf develop the notion of commutative algebra
  objects in a presentably symmetric monoidal $\tB$-category $\cC$. We will use the following
  facts:
  \begin{enumerate}
      \item For any $\cC\in \CAlg(\pPrl{\tB})$, there exists a co-complete $\tB$-category
      $\CAlg(\cC)$ together with, for all $A\in \tB$ a natural identification
      $\CAlg(\cC)(A)=\CAlg(\cC(A))$. 
      \item For any commutative algebra $R\in \CAlg(\cC)$, there exists a $\tB$-category
      $\Mod{\cC}{R}$, together with, for any $p\colon A\to \pt \in \tB$ a natural identification
      $\Mod{\cC}{R}(A)=\Mod{\cC(A)}{p^{\ast}R}$. Moreover, 
      by~\cite[\href{https://arxiv.org/pdf/2209.05103v2\#equation.2.5.2.8}{Prop.~2.5.2.8}]{mwpresentable} the $\tB$-category $\Mod{R}{\cC}$ is again
      presentably symmetric monoidal.
  \end{enumerate} 
\end{rem}

\begin{construction} \label{constr: units}
Let $\CAlggrp(\Omega_{\tB})\subseteq \CAlg(\Omega_{\tB})$ denote the full
$\tB$-subcategory spanned by the \textit{grouplike} commutative monoid objects 
          in $\Omega_{\tB}$. The inclusion admits a parametrized right adjoint 
          \[ (-)^{\times} \colon \CAlg(\Omega_{\tB}) \longrightarrow \CAlggrp(\Omega_{\tB})  \] sending a commutative algebra to its \mdef{units}.  Both, the inclusion and its right adjoint are computed sectionwise, as they preserve limits.  

 Let $\Sp^{\rm{cn}} \in \CAlg(\Prl)$ denote the category of connective spectra. In accordance
  with \cref{defn: param sheaves}, we define 
  the $\tB$-category of \tdef{connective $\tB$-spectra} as  
  \[ \mdef{\pSpcn{\tB}}:= \Omega_{\tB} \otimes \Sp^{\rm{cn}} \in \CAlg(\pPrl{\tB}).\]

  Applying the Boardman--Vogt--May theorem \cite[\href{https://www.math.ias.edu/~lurie/papers/HA.pdf\#theorem.5.2.6.10}{Thm.~5.2.6.10}]{HA}
  sectionwise gives an equivalence of $\tB$-parametrized categories
          \[ \CAlggrp(\Omega_\tB)\simeq \pSpcn{\tB}.\]
          Under this identification, we have an evident forgetful functor to the universe,
          which we denote by
          \[ \ulSpc{} \colon \pSpcn{\tB}\too \Omega_{\tB}\,. \]
\end{construction}

\begin{war}
Beware that a sheaf of connective spectra is not the same as a sheaf of spectra
          whose sections happen to be connective, as the sheaf condition is 
          a limit formula. Nonetheless, there is a full faithful functor 
          $\pSpcn{\tB} \to \Sp_{\tB}$ given sectionwise by including and then sheafifying,
           with right adjoint given by taking connective covers, see also the
           $t$-structure discussion in \cref{ssec: Gluing the Fourier}.
\end{war}

The sectionwise loop and suspension functors induce an adjunction
          \[ \Sigma\colon \pSpcn{\tB} \adj \pSpcn{\tB}\noloc \Omega\]
          where $\Sigma$ is computed by suspending in presheaves and then sheafifying,
          while $\Omega$ is just computed in presheaves of connective spectra. In particular,
          we have that $\Omega\circ \Sigma \simeq \id$, i.e.~$\Sigma$ is fully faithful.

          Since $\pSpcn{\tB}$ is presentably symmetric monoidal by construction, 
          so is $\Mod{\pSpcn{\tB}}{A}$ for any $A\in \CAlg(\pSpcn{\tB})$ 
          by \cref{rem: param modules + algebras}. In particular, these categories admit an
          internal mapping object which for $M,N\in \Mod{\pSpcn{\tB}}{A}$ we
          denote by
          \[ \map_{A}(M,N)\in \Mod{\pSpcn{\tB}}{A}.\]

\section{The Fourier transform for motivic K-theory}\label{sec:Fourier}

Given a ring $R$ and a finite abelian group $G$ with Pontryagin dual $G^\vee$,
the classical Fourier transform is a natural map of commutative rings:
\[\mathcal{F}_\omega:R[G]\longrightarrow R^{G^\vee}\]
which depends on the choice of an orientation, i.e.~a map of abelian groups
\[\omega:\mathbb{Q}/\mathbb{Z}\longrightarrow R^\times.\]
There are numerous generalizations, categorifications and derived enhancements of this
construction. The goal of this section is to produce two more. First,
a Fourier--Mukai transform for modules over the category of perfect complexes on a scheme. Secondly, we use this to produce a Fourier transform
in the category of $\KGL$-modules in the stable motivic category $\pSH$.

We begin in \cref{ssec: parametrized Fourier} by generalizing the Yoga 
of~\cite{chromfour} to the $\tB$-parametrized setting for any topos $\tB$. 
Concretely, setting up a Fourier transform has two inputs:
\begin{enumerate}
    \item A pair $(R,I)$ consisting of a base ring $R\in \CAlg(\pSpcn{\tB})$ together with
          a dualizing object $I\in \Mod{\pSpcn{\tB}}{R}$ which we call a \tdef{Fourier context}.
     \item A presentably symmetric monoidal $\tB$-category $\cC$, together with a
     \tdef{Fourier orientation}, given by a map of connective $\tB$-spectra
     \[ \mdef{\omega} \colon I \too \1_{\cC}^{\times}\,.\]
\end{enumerate}

Using the parametrized adjunction 
\[ \1_{\cC}[-] \colon \Mod{\pSpcn{\tB}}{R} \rightleftarrows \CAlg(\cC)\noloc (-)^{\times}\]
and writing $I(-)= \map_R(-,I)$ for the internal mapping object, for any
$M\in \Mod{\pSpcn{\tB}}{R}$, the data above gives rise a natural map
\[ \mdef{\Four_{\omega}^{M}}\colon \1_{\cC}[M] \too \1_{\cC}^{\ulSpc{I(M)}} \in \CAlg(\cC)\,,\]
which we call the \tdef{Fourier transform} associated to the datum $(R,I,\cC,\omega)$.

In \cref{ssec: cat Fourier}, we apply this formalism to construct a compact object preserving Fourier-Mukai transform.
Concretely, we work over the topos $\Ettop{\base}$ of \'etale
sheaves of spaces on $\Sch_{\base}$ for some base scheme $\base$.
Our chosen Fourier context is  $(\ZZ_{\base},\Sigma \Gm)$, where $\ZZ_{\base}$ is 
the constant sheaf and $\Sigma \Gm$ is the sheaf of connective
$\ZZ$-modules represented by the \'etale delooping of the commutative group scheme $\Gm$.

The functor taking a stack $\stack \in \Stk{\base}$ to its category 
of perfect complexes $\Perf(\stack)$ defines an \'etale sheaf on $\Sch_{\base}$ with values
in $\Catperf$ and we denote the associated $\Ettop{\base}$-category by $\mdef{\pQCoh{\base}}$.
Then, the parametrized category of modules over $\pQCoh{\base}$ comes equipped with a 
natural orientation map 
\[\orcan\colon\Sigma \Gm \too (\pQCoh{\base})^{\times}\]
induced by the inclusion $\Perf(\base)^{\heart}\to \Perf(\base)$. 
For any $\ZZ_{\base}$-module $M$ we write $M^{\vee}=\map_{\ZZ_{\base}}(M,\Sigma \Gm)$
which we call the \tdef{shifted Cartier dual} of $M$.
The formalism of \cref{ssec: parametrized Fourier} then provides us with a 
natural map
\[\Four_{\orcan}^{M} \colon \pQCoh{\base}[M] \too \left(\pQCoh{\base}\right)^{\ulSpc{M^{\vee}}},\]
of commutative algebra objects in the category $\pMQC{\base}$ 
which we call the \tdef{categorical Fourier transform}. 
In particular, on global sections
$\Four_{\orcan}^{M}$ defines an exact, symmetric monoidal, left adjoint 
functor between small, stable, idempotent complete, symmetric monoidal categories. 

In \cref{ssec: motivic Fourier}, we use the categorical Fourier transform to construct
a Fourier transform in the category of modules over the motivic $K$-theory spectrum 
$\KGL_{\base}\in \SH{\base}$ by applying $K$-theory sectionwise and then localizing
into the category of $\KGL_{\base}$-modules in $\SH{\base}$. More precisely, if $M$
is an \'etale sheaf of connective $\ZZ$-modules, such that the dual $M^\vee$ is 
representable by a \emph{stack} in the sense of \cref{ter: stack},
we construct natural map of $\KGL_{\base}$-algebras in $\SH{\base}$
\[ \Four_{\rm{mot}}^{M}\colon \KGL_{\base}[M] \too \KGL_{\base}^{\ulSpc{M^{\vee}}},\]
which we call the \tdef{motivic Fourier transform}. Here, the left hand side is the
$\KGL$-homology of $M$ defined via Kan-extension from smooth representables and the 
right hand side is the $\KGL$-motive of the stack $\Omega^\infty M^{\vee}$, 
as defined in \cref{not: KGL cohom}.

We conclude this section by proving that, for any torus $\t/\base$, the motivic Fourier
transform provides an equivalence 
\[ \KGL_{\base}[\cow{\t}]\xrightarrow{\sim} \KGL_{\base}^{\deloop{\t}},\]
where $\cow{\t} = \deloop{\t}^{\vee}$ is the sheaf of characters $\t \to \Gm$.
This will be a crucial input for proving our main theorem in \cref{sec: Motivic EM}.
          
\subsection{Parametrized Fourier theory}\label{ssec: parametrized Fourier}

  Let $\cC\in \CAlg(\pPrl{\tB})$ be a presentably symmetric monoidal $\tB$-category.
  Then $\cC$ comes equipped with a unit map
  \[\Omega_\tB\too \cC \]
  which is left Kan extended from the functor $\pt \xrightarrow{\1_{\cC}} \cC$. 
  Taking commutative algebras objects, we obtain a map of presentable $\tB$-categories
  \[ \mdef{\1_{\cC}[-]} \colon \pSpcn{\tB} 
   \hookrightarrow \CAlg(\Omega_{\tB}) \too \CAlg(\cC) \]
   which we call the \tdef{group algebra} construction. We denote the parametrized
   right adjoint by
   \[ \mdef{(-)^{\times}}\colon \CAlg(\cC)\too \pSpcn{\tB} \,, \]
  and refer to it as the \tdef{spectrum of units} functor.
   Explicitly $(-)^{\times}$ is given by the composition of the mapping object 
   $\map_{\cC}(\1_{\cC},-)$ with the functor of \Cref{constr: units}.
   
  Dually, by parametrized right Kan extension of the map 
  $\pt \xrightarrow{\1_{\cC}} \CAlg(\cC)$ in $\Cat(\tB)$, we obtain a right
  adjoint functor
  \[ \mdef{\1_{\cC}^{(-)}}\colon \Omega_{\tB}^{\op}\too \CAlg(\cC) \,.\]

  For any $A\in \tB$, we think of $\1_{\cC}[A]$ as the \tdef{homology}
  and of $\1_{\cC}^{A}$ as the \tdef{cohomology} of $A$ with coefficients in $\cC$.
  Note that this is consistent with the notation of \cref{example: parametrized colimit}.

\begin{defn}\label{defn: Four orientation}
A \tdef{Fourier context} is a pair \mdef{$(R,I)$} consisting of a \textit{base ring}
$R\in \CAlg(\pSpcn{\tB})$ and a \textit{dualizing object}
$I\in \Mod{\pSpcn{\tB}}{R}$. For any $\cC\in \CAlg(\pPrl{\tB})$ a
\tdef{Fourier orientation}\footnote{Called a \emph{pre}-orientation in~\cite{chromfour}.}
of $\cC$ with respect to $(R,I)$ is a map of parametrized connective spectra
  \[ \mdef{\omega}\colon I \too \1_{\cC}^\times.\] 
  For any $R$-module $M\in \Mod{\pSpcn{\tB}}{R}$ we write
  \[ \mdef{I(M)}\coloneq \map_R(M,I) \in \Mod{\pSpcn{\tB}}{R}\]
  for the internal mapping object in connective $\tB$-spectra.
\end{defn}
    
The following is the parametrized analog of~\cite[\href{https://arxiv.org/pdf/2210.12822\#thm.3.10}{Prop.~3.10}]{chromfour} and the proof is essentially the same.
\begin{prop}\label{Fourier thm}
  Let $\cC\in \CAlg(\pPrl{\tB})$ and let $(R,I)$ be a Fourier context.
  There is a natural equivalence
  between the space of parametrized natural transformations
  \[\1_\cC[-]\longrightarrow\1_\cC^{\ulSpc{I(-)}} ,\] 
  of parametrized functors $\Mod{R}{\Spcn_{\tB}} \to \CAlg(\cC)$ and the space of
  Fourier orientations, i.e.~the space of maps of parametrized connective spectra
  \[  \omega \colon I \too \1_\cC^\times\,.\]
\end{prop}
\begin{proof}
We claim that the parametrized right Kan extension of the functor 
$\ast \xrightarrow{\1_{\cC}} \CAlg (\cC)$ along the map
$\ast \xrightarrow{I} \Mod{\pSpcn{\tB}}{R}$ is given by the functor sending 
$M$ to $\1_{\cC}^{\ulSpc{I(M)}}$. 
As $\ast \to \Omega_{\tB}$ is fully faithful, the latter map factors through the map $\Omega_{\tB}^{\op} \to \Mod{\pSpcn{\tB}}{R}$ obtained by parametrized right Kan extension. 
But this map has a parametrized left adjoint, given by sending $M$ to $\ulSpc{\map_R(M,I)}$.
The claim follows because parametrized right Kan extension along a parametrized right adjoint is computed by precomposition with the left adjoint. 

Using the universal property of the parametrized right Kan extension, we finally deduce that

 \[ \Map_{\Fun(\Mod{\pSpcn{\tB}}{R},\CAlg(\cC))}\left(\1_{\cC}[-], \1_{\cC}^{\ulSpc{I(-)}}\right)
 \simeq \Map_{\CAlg(\cC)}(\1_{\cC}[I],\1_\cC) \,. \] 
  The claim follows as the spectrum of units was defined to be the right adjoint of the group algebra functor.
\end{proof}

\begin{defn}\label{defn: Fourier}
  Let $(R,I)$ be a Fourier context in a parametrized
  presentably symmetric monoidal category $\cC\in \CAlg(\pPrl{\tB})$. By \cref{Fourier thm}, any orientation
  $\omega\colon I \to \1_{\cC}^\times$ yields an associated natural transformation
  \[ \mdef{\Four_{\omega}^{(-)}}\colon \1_{\cC}[-]\too \1_{\cC}^{\ulSpc{I(-)}}\] of functors
  $\Mod{\pSpcn{\tB}}{R}\to \CAlg(\cC)$. We call $\Four_{\omega}$ the \tdef{Fourier
    transform} associated to $\omega$.
\end{defn}

\subsection{A Fourier--Mukai transform for \texorpdfstring{$\Perf$}{Dperf}-linear categories}\label{ssec: cat Fourier}

Throughout this section we fix a base scheme $\base$
and work over the topos of \'etale sheaves on the category of $\base$-schemes
$\mathrm{Sch}_{\base}$ which we denote as
\[ \mdef{\Ettop{\base}}\coloneq\Shv_{\et}(\Sch_{\base};\Spc)\,.\] 
Note that, since any stack over $\base$ is \'etale locally representable, 
we also have a natural equivalence $\Ettop{\base}\simeq \Shv_{\et}(\Stk{\base};\Spc)$
so it does not matter which site we use to produce our base topos.

We now use the machinery developed in the previous section to set up a
Fourier--Mukai transform for the category of perfect complexes.

Recall from the discussion in \cref{ssec: mot k theory} that for any
$\stack \in \Stk{\base}$ we have a well behaved notion of perfect complexes on $\stack$,
namely the category $\Perf(\stack)= \QCoh(\stack)^{\omega} \in \Catperf$.
By~\cite[Cor.~4.25]{bgt13} the category $\Catperf$ is compactly generated and presentable.
Moreover, the inclusion $\Catperf \simeq \cPrlst \to \Prlst$ preserves colimits and is closed
under tensor products and so $\Catperf$ inherits the structure of a presentably 
symmetric monoidal category and the forgetful functor $\Catperf \to \Cat$ preserves
limits. In particular, the functor $\Perf(-)$ refines to an \'etale sheaf of 
symmetric monoidal categories, which we record in the following definition.

\begin{defn}\label{defn: Qcoh parametrized}
For any scheme $\base$, we write 
  \[ \mdef{\pQCoh{\base}} \in \CAlg(\Omega_{\Ettop{\base}} \otimes \Catperf)\]
 for the sheaf of categories given by the functor
  \[\pQCoh{\base}\colon \Stk{\base}^{\op}\too \CAlg(\Perf), \quad X \mapsto 
  \Perf(X).\]
  Moreover, we denote the $\Ettop{\base}$-parametrized category of modules over 
  $\pQCoh{\base}$ as
  \[ \mdef{\pMQC{\base}} \coloneq 
  \Mod{\Catperf \otimes \Omega_{\Ettop{\base}}}{\pQCoh{\base}}
  \in \CAlg (\pPrl{\Ettop{\base}})\]
    and refer to sections of $\pMQC{\base}$ as \tdef{\'etale $\base$-linear categories}.
\end{defn}

Explicitly, the parametrized category $\pMQC{\base}$ sends $A \in \Ettop{\base}$
to the presentably symmetric monoidal category 
\[\MQC(A)= \Mod{\Ettop{A} \otimes \Catperf}{\pQCoh{A}}\,.\]
Note that, by \cref{rem: param modules + algebras} the fact that the category
$\Catperf$ is presentably symmetric monoidal, implies that the category 
$\pMQC{\base}$ is \emph{parametrized} presentably symmetric monoidal.

\begin{war}
Note that, if $A\in \Ettop{\base}$ is not representable by a stack, then the
category $\Gamma_{A}\pQCoh{\base}$ as defined by the Kan extension above need not agree with
the category of compact objects of $\QCoh(A)$ as defined via Kan extension of $\QCoh(-)$
as a functor into $\Prlst$. 
\end{war}
Let $f\colon A\to \base$ be an object of $\Ettop{\base}$. 
Since the category $\pMQC{\base}$ is parametrized presentable, the pullback functor 
\[ f^{\ast}\colon \MQC(\base)\too \MQC(A)\]
admits both adjoints $f_{\sharp} \dashv f^{\ast} \dashv f_{\ast}$ which are worth
discussing explicitly. First, note that $f^{\ast}$ is given by
restriction along the pushforward map $(\Ettop{\base})_{/A} \too \Ettop{\base}$.
Thus, the functor $f_{\ast}$ is computed by right Kan extension
along the same functor and is thus given by precomposition with the left adjoint.
In particular, for any $\cD \in \MQC(\fZ)$ we have that $f_{\ast}\cD$ 
is the sheaf of $\pQCoh{\base}$-modules given by
\[ f_{\ast}\cD(B)= \cD(f^{\ast}B)=\cD(A\times_\base B) \in \Mod{\Catperf}{\pQCoh{\base}(B)}.\]
The left adjoint $f_{\sharp}$ is computed by left Kan extension over the slice, which 
is in general mysterious. 
However, if $f$ is finite \'etale, the adjoints will agree in any 0-semiadditive category
as the adjoints are locally given by product and coproduct. Let us call a 
$\Ettop{\base}$-category $\cC$ \tdef{finite \'etale (co)-complete} if it is (co)-complete
with respect to the inductible class of finite \'etale morphisms in the sense 
of \cite[\href{https://arxiv.org/pdf/2403.07676\#definition.2.10}{Def.~2.10}]{ppp}.

\begin{lem}\label{lem: fet semi-add}
    Let $\cC$ be a finite \'etale complete and cocomplete $\Ettop{\base}$-category that is
    sectionwise $0$-semiadditive. Then for any finite \'etale map
    $f\colon X \to \base$ we have a natural equivalence
    $\rm{Nm}_{f}\colon f_{\sharp}\xrightarrow{\sim} f_\ast$ of functors
    $\cC(X)\to \cC(S)$.
\end{lem}
\begin{proof}
    As any $(-1)$-truncated map in $\Ettop{\base}$ is either empty or an 
    equivalence and $\cC$ is sectionwise pointed, we may construct a comparison
    map 
    \[ \mathrm{Nm}_{f} \colon f_{\sharp} \too f_{\ast},\]
    as in \cite[\href{https://arxiv.org/pdf/2403.07676\#construction.3.3}{Cons.~3.3}]{ppp}
    inductively. As $f$ is finite
    \'etale, it is \'etale locally given by a fold map
    $\sqcup_n S \to \base$ for some finite $n$. Thus, as $\cC$ is finite \'etale
    complete and cocomplete, we may use the $(-)_\ast$ and $(-)_{\sharp}$ 
    Beck--Chevalley isomorphisms to see that $\mathrm{Nm}_{f}$ is \'etale locally
    given by the identity matrix 
    \[ \coprod_n S \too \prod_{n} S\,,\] 
    and thus an equivalence by the assumption that $\cC$ is 0-semiadditive. 
    As we can check equivalences \'etale locally, the claim follows.
\end{proof}

In particular \cref{lem: fet semi-add} applies to $\pMQC{\base}$ as it is even
parametrized presentable and sectionwise $0$-semiadditive, as well as any sheaf category
$\Ettop{\base} \otimes \cC$ where $\cC$ is presentable and 0-semiadditive.

Let us now define our chosen Fourier context.
Recall that we denote
by $\ZZ_{\base}\in \pSpcn{\Ettop{\base}}$ the \'etale sheafification of the constant presheaf
with value $\ZZ$. We write 
\[ \mdef{\ZMod{\base}}\coloneq \Mod{\pSpcn{\Ettop{\base}}}{\ZZ_{\base}} \]
for the parametrized category of $\ZZ_{\base}$-modules. 

\begin{notation}
Let $\cE$ be a commutative group scheme over $\base$. Then, via Yoneda, we can
regard $\cE$ as an \'etale sheaf of abelian groups on $\Sch_{\base}$ and thus via
the Dold--Kan correspondence as an object of $\ZMod{\base}$. As such, we have a
canonical identification
\[ \ulSpc{\Sigma\cE} \simeq \deloop{\cE} \in \Ettop{\base}\]
where $\deloop{\cE}$ denotes the \'etale classifying stack. We make this identification
implicitly during the rest of this section.
\end{notation}

\begin{defn}\label{defn: picard}
    For any $M\in \ZMod{\base}$, we call the internal mapping objects
    \[ \mdef{\cow{M}}\coloneq \map_{\ZZ_{\base}}(M,\Gm) \in \ZMod{\base}\]
  \[ \mdef{M^{\vee}} \coloneq \map_{\ZZ_{\base}}(M,\Sigma \Gm) \in \ZMod{\base}\]
 the \tdef{Cartier dual} and \tdef{shifted Cartier dual} of $M$ respectively.
\end{defn}

\begin{rem}\label{defn: char lattice}
Let us say some words about these two notions of duality.
\begin{enumerate}
    \item By definition, classes in the Cartier dual 
          $\Lambda(M)$ correspond to characters maps $\chi \colon M \to \Gm$. 
          If $M$ is given by an algebraic torus $\t$, then after passing to a finite \'etale
          extension of $\base$ that splits $\t$ the sheaf $\cow{\t}$ is constant and
          given by a free abelian group of finite rank, i.e.~a lattice. Thus, in
          this situation, we also call $\cow{\t}$ the \tdef{character lattice of $\t$}.
    \item Using the identification $\Sigma \Gm \simeq \pic_{\base}$ of \cref{lem: BGm is pic},
          we see that classes of the shifted Cartier dual $M^{\vee}$ are given by rank $0$ line bundles $\cL\in \pic_{\base}(M)$. Thus, if $M=\cE$ is an abelian
          variety, the shifted Cartier dual $\cE^{\vee}$ is representable and recovers the dual abelian variety
          of $\cE$.
\end{enumerate}
\end{rem}

    Since the \'etale suspension functor is fully faithful, it provides a
    canonical identification
    \[\cow{M}\simeq \map_{\ZZ_{\base}}(\Sigma M, \Sigma \Gm)= \pic_{\base}(\Sigma M)
    = (\Sigma M)^{\vee} \in \ZMod{\base},\]
    which we will use freely. Moreover, for any $M\in \ZMod{\base}$, the equivalence
    \[\cow{M} \simeq \Omega \map_{\ZZ_{\base}}(M, \Sigma \Gm) =\Omega M^{\vee}, \]
    induces a natural comparison map 
    \[ \Sigma \cow{M} \too M^{\vee}.\]
    
\begin{lem}\label{lem: torus dual computation}
  Let $\t$ be a torus over $\base$ and consider $\t$ as an object of $\ZMod{\base}$. 
  The natural maps 
  \[ \Sigma \cow{\t}\too \t^{\vee}\]
  \[  (\t^{\vee})^{\vee}\too \t\]
  \[  ((\Sigma\t)^{\vee})^{\vee}\simeq \cow{\t}^\vee \too \Sigma{\t},\]
  are equivalences of \'etale connective $\ZZ_{\base}$-modules.
\end{lem}
\begin{proof}
 Since both sides are \'{e}tale sheaves, we can check this \'etale locally, so we
 can assume that $T\simeq \Gm^{\times n}$ is split. Moreover, since
everything in sight commutes with products, we can take $n=1$. 
The first claim is then that the map 
\[ \Sigma \ZZ_{\base} \too \map_{\ZZ_{\base}}( \Gm, \Sigma \Gm) = \pic_{\base}(\Gm)\]
is an equivalence, which follows since every line bundle over $\Gm$ is trivial.
Using the first equivalence, the second map is identified with the composite
 \[ (\Sigma \ZZ_{\base})^\vee = \map_{\ZZ_{\base}}(\Sigma \ZZ_{\base}, \Sigma{\Gm})
 \simeq \map_{\ZZ_{\base}}(\ZZ_{\base}, \Gm) \simeq \Gm,\]
where we have again used fully faithfulness of \'etale suspension. 
Finally, the third map simply becomes the canonical equivalence
 \[ \ZZ_{\base}^\vee = \map_{\ZZ_{\base}}(\ZZ_{\base}, \Sigma{\Gm})\simeq \Sigma{\Gm}\]
 and so we are done.
\end{proof}

Recall that, since $\pMQC{\base}$ is presentably symmetric monoidal, we have a group ring functor
\[ \pQCoh{\base}[-] \colon \pSpcn{\Ettop{\base}} \too \CAlg(\pMQC{\base}),\]
which admits a parametrized right adjoint $(-)^{\times}$. Unwinding the definitions,
we see that any $\cC\in\CAlg(\pMQC{\base})$ has an underlying sheaf 
with values in $\CAlg(\Catperf)$ and the sheaf of units 
$\cC^{\times}$ is sectionwise given by the Picard spectrum
\[ \cC^{\times}(A) \simeq \pic(\cC(A)) \in \Spcn\,,\]
i.e.~the grouplike $\EE_\infty$-space of $\otimes$-invertible objects in $\cC(A)$.
For any $X\in \Sch_{\base}$, equip the category $\Perf(X)$ with the standard
$t$-structure, see also the discussion in \cref{ssec: Gluing the Fourier},
such that the heart $\Perf(X)^{\heart}$ is an abelian 
$1$-category, equipping $\pic(\Perf(X)^{\heart})$ with a natural $\ZZ$-linear structure.

We write $\mdef{\pic_{\base}} \in \ZMod{\base}$ for the \'etale sheaf which takes 
$X\in \Sch_{\base}$ to the $\ZZ$-linear connective spectrum $\pic(\Perf(X)^{\heart})$. 
Observe that, for any any $A\in \Sch_{\base}$ we have natural equivalences of abelian groups
\begin{equation}\label{eq: loop pic}
  \Omega \pic_\base(A) \simeq (\Omega\QCoh(A)^{\heartsuit})^\times
  \simeq (\cO_{A})^\times \simeq \Gm(A) 
\end{equation}
Moreover, we can actually de-loop this identification.

\begin{lem}\label{lem: BGm is pic}
  The natural map  $\Sigma \Gm \too \pic_\base$
  induced by~\eqref{eq: loop pic} is an equivalence of \'etale 
  $\ZZ$-linear connective spectra.
\end{lem}
\begin{proof}
  It follows from the discussion above that the natural map $\Sigma \Gm \too \Sigma \Omega \pic_\base$ is an equivalence,
  so it suffices to argue that the counit $\Sigma \Omega \pic_\base \to \pic_\base$ is an equivalence.
  This is true because every line bundle is
  \'etale locally trivial, so that $\pic_\base$ is \'etale locally connected.
\end{proof}

The observations above allows us define the desired orientation of the parametrized 
category $\pMQC{\base}$. In particular, we get an associated Fourier transform
as in \cref{defn: Fourier}.

\begin{defn}\label{defn: can orientation + fourier}
  The \tdef{canonical orientation} of the $\Ettop{\base}$-category
  $\pMQC{\base}$ in the Fourier context $(\ZZ_{\base},\Sigma \Gm)$ is given
  by the natural map of \'etale connective spectra
  \[ \mdef{\orcan} \colon \Sigma \Gm \too (\pQCoh{\base})^\times \]
  which sectionwise on $X\in \Sm{\base}$ is induced by the identification 
  $\Sigma \Gm \simeq \pic_{\base}$ combined with the inclusion of the heart
  $\Perf(X)^{\heartsuit}\to \Perf(X)$.
  For any $M\in \ZMod{\base}$, we call the associated Fourier transform
    \[\mdef{\Four_{\orcan}^{M}}\colon \pQCoh{\base}[M] 
    \too (\pQCoh{\base})^{\ulSpc{M^{\vee}}}\in \CAlg(\pMQC{\base})\]
    the \tdef{categorical Fourier transform}.
\end{defn}

Because the group ring construction $\pQCoh{\base}[-]$ is difficult
to identify explicitly, the main purpose of the Fourier transform
of \cref{defn: can orientation + fourier} is to help us construct the 
motivic Fourier transform in \cref{ssec: motivic Fourier}. The one
crucial thing we verify however is that, for a torus $\t$ it recovers
the decomposition of the category of representations of $\t$ into
characters. This is the content of the following proposition.

\begin{prop}\label{prop: Fourier computation QCoh}
  For any torus $\t$ over $\base$, the categorical Fourier transform
  \[ \Four_{\orcan}^{\cow{\t}}\colon \pQCoh{\base}[\cow{\t}] 
  \too (\pQCoh{\base})^{\deloop{\t}}\in \CAlg(\pMQC{\base})\]
  is an equivalence of $\Ettop{\base}$-parametrized categories.
\end{prop}
\begin{proof}
First, note that by \cref{lem: torus dual computation} we have 
$\cow{\t}^{\vee}\simeq \deloop{\t}$, so the target of the Fourier
transform is indeed given by $\left(\pQCoh{\base}\right)^{\deloop{\t}}$

Since equivalences can be checked \'etale
locally and both sides take products of \'etale connective spectra to 
tensor products, we may reduce to the case
$\t=\Gm$ and $\cow{\t}=\ZZ$. It thus suffices to show that, for any representable
$A\in \Sch_{\base}$, the $\Perf(A)$-linear functor
  \[ \Gamma_A(\Four_{\orcan}^{\ZZ})\colon \Gamma_A(\pQCoh{\base}[\ZZ])
  \too \Gamma_A(\pQCoh{\base})^{\deloop{\t}} = \Perf(\deloop{\GG_{m,A}})\]
  is an equivalence. 
  Since in $\Catperf$ the map from the infinite coproduct to the infinite product
  is fully faithful and hence a monomorphism, the colimit over the discrete space
  $\ZZ$ in the sheaf category $\Gamma_A(\pMQC{\base})$ is computed sectionwise.
  Thus, the sections of $\pQCoh{\base}[\ZZ]$ are given by the coproduct
  \[ \Gamma_A(\pQCoh{\base}[\ZZ]) \simeq \Perf(A)[\ZZ] 
  \simeq \coprod_{\ZZ}\Perf(A) \in \Catperf.\] 
  Moreover, since the composite $\Catperf \simeq \Prlc\to \rm{Pr}^{\rm{L}}$
  preserves colimits, the symmetric monoidal category $\Perf(A)[\ZZ]$ naturally
  identifies with compact objects in the category of $\ZZ$-graded objects
  $\QCoh(A)^{\mathrm{gr}}$ carrying the Day-convolution symmetric monoidal structure. 
  
  Unwinding the definitions, we learn that the Fourier transform
  $\Gamma_A(\Four_{\orcan}^{\ZZ})$ is determined by the map of connective spectra
  \[\ZZ\too \Perf(\deloop{\GG_{m,A}})^{\times}\,\] 
  picking out the canonical line bundle $\cO(1)$. Thus, under the identification
  of the left hand side described above, the functor 
  \[\Gamma_A(\Four_{\orcan}^{\ZZ})\colon \QCoh(A)[\ZZ] \too \QCoh(\deloop{\GG_{m,A}})\] 
  takes a graded complex of sheaves $(E_n)_{n\in \ZZ}$ to the complex 
  $\bigoplus_{n\in \ZZ} E_n$ with the natural $\Gm$-action. This is well known 
  to be an equivalence of categories, for example
  by~\cite[\href{https://arxiv.org/pdf/1907.13562\#equation.4.1}{Theorem 4.1}]{Moulinos_2021}, with inverse
  taking $F\in \Perf(\deloop{\Gm})$ to the graded complex 
  $(\Gamma(B\GG_m,F\otimes \cO(n)))_{n\in\ZZ}$ and so we conclude.
\end{proof}


\subsection{Assembling the motivic Fourier transform}\label{ssec: motivic Fourier}

Fix $\base$ to be an affine, regular base scheme.
We now use the parametrized Fourier transform constructed in \cref{defn: Fourier}
to build a Fourier transform in the category of $\KGL_{\base}$-modules, which 
we call the \tdef{motivic Fourier transform}.

Let $\stack \in \Stk{\base}$ be a stack and consider the lax symmetric monoidal functor
\[ K_{\stack}\colon \MQC(\stack) \too \Shv_{\et}(\Stk{\stack};\Catperf) \too \Shv_{\et}(\Sm{\stack};\Catperf) \xrightarrow{K} \Fun(\Sm{\stack}^{\op},\Sp)\,\]
given by first forgetting to the underlying $\Catperf$-valued sheaf, then restricting to smooth stacks over $\stack$, followed by applying
$K$-theory sectionwise (which generally might not preserve the sheaf condition).
For any $E\in \SH{\stack}$ we write $\ulPsh{\stack}$ for the sheaf of
spectra on $\Sm{\stack}$ given by the assignment taking $Y\in \Sm{\stack}$
to the mapping spectrum $\map(\SS_{\stack}[Y], E)$. 
We begin by observing the following.

\begin{lem}
   For any regular stack $\stack$, there is a natural equivalence 
    \[  K_{\stack}(\pQCoh{\stack}) \xrightarrow{\sim}
    u_{\stack}(\KGL_\stack) \, \]
   of presheaves of spectra on $\Sm{\stack}$.

\end{lem}
\begin{proof} 
Let $p\colon Y\to \stack\in \Sm{\stack}$ be smooth representable. By construction,
 $\KGL_{\stack}$ has underlying (pre)-sheaf given by
\[ u_{\base}(\KGL_{\stack})(Y) = L_{\AA^1}K(\Perf(Y))\,.\]
Since $\stack$ is regular and $p$ is smooth, $Y$ is regular as well and hence we get that
the natural map
\[ L_{\AA^1}K(\QCoh(\stack_Y))\xrightarrow{\sim} K(\Perf(\stack_Y))= K_{\stack}(\pQCoh{\stack})(Y)\]
is an equivalence, as claimed.
\end{proof}
In particular, we learn that we have $K_{\stack}(\pQCoh{\stack}) \simeq \KH_\stack$, 
so that $K_\stack$ naturally lifts to a functor to modules over $\KH_\stack$
in presheaves of spectra. We denote by
\[ 
\mdef{\KFun{\stack}} \colon 
\MQC(\stack) \longrightarrow \Mod{\SH{\stack}}{\KGL_\stack}
\]
the functor given by the localization $L_{\beta}L_{\AA^1}L_{\Nis}K_\stack$, that is we apply $K$-theory sectionwise and then apply Nisnevich localization, $\AA^1$-localization and Bott-periodization to the resulting presheaf.
All the functors are compatible with representable $(-)^{\ast}$, so that $\KFun{\stack}$
is as well. This allows us to define assembly and coassembly maps for representable
morphisms. In fact we can use \cref{prop:Bott-periodization-is-localization} to get a
coassembly equivalence for all maps of stacks, not just representable ones.

\begin{lem}\label{lem: computation target fourier}
    For any stack $f\colon \fX \to \base$, 
    we have a natural equivalence of $\KGL_{\base}$-algebras
    \[ \KFun{\base}(f_{\ast}\pQCoh{\stack}) \xrightarrow{\sim} f_{\ast} \KGL_{\stack}\,.\]
\end{lem}
\begin{proof}
Unwinding the definition of $\KFun{\base}$, we see that the
   left hand side is the Bott-periodization of the motivic $S^1$-spectrum which
   takes a smooth representable $p\colon Y\to \base$ to
   \[ L_{\AA^1}L_{\Nis}K(\Perf(Y\times_{\base} \stack)) = \KH_{\base}(Y\times_{\base} \stack)\,,\]
   which by \cref{war: completed KGL} is naturally equivalent to the underlying
   sheaf of spectra represented by $f_{\ast} \KGL_{\stack}$. In particular,
   it is already a $\KGL_\base$-module and so the claim follows from the fully faithfulness
   of \cref{prop:Bott-periodization-is-localization}.
\end{proof}

The assembly is more mysterious, so we spell it out and upgrade it to a map
of group algebras.

\begin{construction}\label{constr: assembly}
    Let $f\colon X\to \base$ be a smooth $\base$-scheme.
    Applying the functor $\KFun{X}$ to the unit of the $f_{\sharp} \dashv f^{\ast}$ 
    adjunction
    \[ \epsilon \colon \pQCoh{X} \too f^\ast f_{\sharp} \pQCoh{X} \in \MQC(X).\] 
     and using that $\KFun{\base}$ commutes with $f^{\ast}$ induces a natural map of $\KGL_{X}$-modules
    \[ \KFun{X}(\epsilon) \colon \KFun{X}(\pQCoh{X}) 
    \too f^\ast \KFun{\base}(f_\sharp \pQCoh{X}) \in \Mod{\SH{X}}{\KGL_{X}}\,. \] 
    Using that $\KFun{X}(\pQCoh{X})\simeq \KGL_X$ by 
    \cref{lem: computation target fourier} and passing to the adjoint yields a map
    \[ \ass_{X} \colon f_\sharp \KGL_X \too \KFun{\base}(f_\sharp \pQCoh{X}) \in \Mod{\SH{\base}}{\KGL_{\base}}.\]
    By left Kan extending in the source, we obtain for every $A\in \Nistop{\base}$
    a natural map 
    \[ \mdef{\ass_{A}} \colon \KGL_{\base}[A] \too \KFun{\base}(\pQCoh{\base}[v^{\ast}A]) \]
    which we call the \tdef{assembly map} of $A$. Here, $v^{\ast}$ denotes
    \'etale sheafification. By monoidality of the adjunctions, $\ass_{(-)}$ enhances 
    to a lax symmetric monoidal transformation
    of lax symmetric monoidal functors and thus refines to a natural transformation
    \[ \ass_{(-)} \colon \KGL_{\base}[-]\too \KFun{\base}(\pQCoh{\base}[v^{\ast}(-)]\]
    of functor $\Nistop{\base} \otimes \Spcn \to \CAlg_{\KGL_{\base}}(\SH{\base})$.
\end{construction}

Since we have little control over the $(-)_{\sharp}$-functors in $\MQC(\base)$, it
is difficult to say things about the assembly in general. However, we know
what happens for finite \'etale maps, which is just enough for our purposes. 

\begin{lem}\label{lem: fet assembly}
    For any finite \'etale scheme $f\colon X\to \base$ the assembly map
    \[ \ass_{X}\colon f_\sharp \KGL_X \longrightarrow \KFun{\base}(f_{\sharp}\pQCoh{X}) \]
    is an equivalence.
\end{lem}
\begin{proof}
  By the naturality of the norm construction, explained
  in~\cite[\href{https://arxiv.org/pdf/2403.07676\#proposition.3.7}{Prop.~3.7}]{ppp}, the diagram
\[ \begin{tikzcd}
    f_\sharp \KGL_X \ar[r] \ar[d] & \KFun{\base}(f_\sharp \pQCoh{X}) \ar[d] \\
    f_\ast \KGL_X  & \ar[l] \KFun{\base}(f_\ast \pQCoh{X})
\end{tikzcd} \]
commutes, where the vertical maps are the norm maps and the vertical once are 
the (co-) assembly maps.
The lower map is an equivalence by \Cref{lem: computation target fourier}, the right hand
map is an equivalence by Atiyah duality \cref{thm: Atiyah Duality} and the right hand map
is an equivalence by \Cref{lem: fet semi-add}, so the claim follows.
\end{proof}

\begin{construction}\label{defn: mot fourier}
    Let $M\in \ZMod{\base}$ such that $\ulSpc{M^{\vee}}$ is representable by a stack. 
    The \tdef{Motivic Fourier transform} associated to $M$ is the map of
    $\KGL_{\base}$-algebras
    \[ \mdef{\Four_{\rm{mot}}^{M}}\colon \KGL_S[M] \too \KGL_S^{\ulSpc{M^{\vee}}}
    \in \CAlg_{\KGL_{\base}}(\SH{\base})\,, \]
    given by the composite map
    \[ \KGL_\base[M] \xrightarrow{\ass_{M}} \KFun{\base}(\pQCoh{\base}[M])
    \xrightarrow{\KFun{\base}(\Four_{\orcan}^{M})} \KFun{\base}((\pQCoh{\base})^{\ulSpc{M^{\vee}}})
    \simeq \KGL_\base^{\ulSpc{M^{\vee}}} \] 
    of the assembly of $M \to S$, the image of the categorical Fourier transform of
    \cref{defn: can orientation + fourier} under the functor $\KFun{\base}$ 
    and the equivalence from \Cref{lem: computation target fourier}.
\end{construction}
We are now ready to prove the first major ingredient for our main theorem, namely
that the representation theoretic equivalence
of~\cref{prop: Fourier computation QCoh} descends to algebraic $K$-theory.
We will first start with the following lemma, giving a more explicit description of the character lattice. 

\begin{construction}\label{lem:lattice-is-finite-etale}
Given a torus $\t/\base$, let $\base^{\prime} \to \base$ be a finite \'etale extension 
that splits $\t$ with Galois-group $\GalG$ and write $\Lambda= \cow{\t}_{\base^\prime}$ 
for the underlying set with $\GalG$-action of $\cow{\t}$. We have a natural 
$\GalG$-equivariant decomposition of $\Lambda$ into the orbits i.e.~an equivalence 
of $\GalG$-sets
\[ \Lambda \simeq \coprod_{[\alpha] \in \Lambda_{\GalG}} \alpha\GalG\,. \]
Since $\cow{\t}$ is locally constant, it is left Kan extended from the small \'etale
site and so by \cref{lem:C-fully-faithful} we obtain an induced decomposition
 of \'etale sheaf of sets
    \[ \cow{\t} \simeq \coprod_{[\alpha] \in \Lambda_{\GalG}} X_{\alpha} 
    \in \rm{Set}^{\et}_{\base},\]
    where $X_{\alpha}\to \base$ is the finite \'etale cover corresponding to the finite,
    transitive $\GalG$-set $\alpha \GalG$.
\end{construction}

With this presentation in hand, we are ready to use the machine of the previous section to 
show that, for a torus $\t$, the motivic Fourier transform witnesses the
character decomposition of the category $\Perf(\deloop{\t})$ on the level of motivic
spectra.

\begin{prop}\label{thm: Fourier computation KGL}
 For any torus $\t$ over $\base$ the motivic Fourier transform
  \[ \Four_{\rm{mot}}^{\cow{\t}} \colon \KGL_{S}[\cow{\t}]
  \too \KGL_{\base}^{\deloop{\t}} \in \CAlg_{\KGL}(\SH{\base})\]
  is an equivalence of commutative $\KGL_{\base}$-algebras.
\end{prop}
\begin{proof}
  By construction, the map $\Four_{\rm{mot}}^{\cow{\t}}$ is given by the composite
  \[ \KGL_{\base}[\cow{\t}] \xrightarrow{\ass_{\cow{\t}}}
  \KFun{\base}(\pQCoh{\base}[\cow{\t}])
  \xrightarrow{\KFun{\base}(\Four_{\orcan}^{\cow{\t}})}
  \KFun{\base}((\pQCoh{\base})^{\deloop{\t}})\,.\]
    As the categorical Fourier transform $\Four_{\orcan}^{\cow{\t}}$ is an equivalence by
    \Cref{prop: Fourier computation QCoh}, it is enough to show that the assembly map for
    $\cow{\t}$ is an equivalence.
    Recall that the functor $\KFun{\base}$ is given by the composite 
    $L_{\beta}L_{\AA^1}L_{\Nis}K_\base$ where $K_{\base}$ is the functor that applies
    $K$-theory section-wise. As $K$-theory is an additive invariant
  by \cite[Thm.~1.3]{bgt13} it commutes with infinite coproducts and so do the localization
  functors as they are left adjoints. Thus, $\KFun{\base}$ commutes with infinite 
  coproducts as well. By \Cref{lem:lattice-is-finite-etale}, we can write
  $\cow{\t}$ as an infinite coproduct of finite étale extensions of $S$ and so the
  claim follows by naturality of the assembly combined with \Cref{lem: fet assembly}.
\end{proof}
\section{Motivic Eilenberg--Moore formulas}\label{sec: Motivic EM} 
Let $\baf$ be a field and write $\base= \Spec{\baf}$ throughout this section. 

We saw in \cref{thm: Fourier computation KGL} that for any torus $\t/\base$, the decomposition
of a representation of $\t$ into characters as in \cref{prop: Fourier computation QCoh}, 
induces an equivalence of $\KGL_{\base}$-algebras
\[ \Four_{\rm{mot}}^{\cow{\t}}\colon \KGL_{\base}[\cow{\t}]
\xrightarrow{\sim} \KGL_{\base}^{\deloop{\t}}.\]
By the computations of \cref{lem: torus dual computation}, we know that the motivic
Fourier transform also gives a natural map 
\[ \Four_{\rm{mot}}^{\Sigma \cow{\t}} \colon \KGL_{\base}[\Sigma \cow{\t}] 
\too \KGL_{\base}^{\t}.\]
The main result of this section, \Cref{thm: mainthm}, is that this map is in fact an equivalence. The proof proceeds along several steps.

In \cref{ssec: Gluing the Fourier}, we discuss how to obtain the second equivalence
as a base change of the first one. The idea is to replace the presentation
of $\t$ via the bar construction, i.e.~as the fiber of the canonical map 
$\base \to \deloop{\t}$, by a more delicate resolution. Namely, we write
$\t$ as a fiber of a map of deloopings of tori $\deloop{\kt}\to \deloop{\wt}$ where $\wt$
is Weil restricted from a split torus. 
This will imply our theorem
once we prove an Eilenberg--Moore formula for the stacks involved, see \cref{algEM},
which occupies us for the remainder of this chapter.

The main important geometric ingredient is an observation which already appeared in the work of Merkurjev and Panin \cite{mp97}, namely that the Weil restricted torus $W$ admits a $W$-equivariant cell decomposition. We will formalize this, by showing that the $W$-equivariant motive of $W$ is \emph{Artin--Tate} and hence \emph{Artin} over $\KGL_{\deloop{\wt}}$.
We will recall those two notions in \cref{ssec: EM}, before turning to the concrete case of a torus in \cref{ssec: EM for tori}.

\subsection{Gluing the motivic Fourier transform}\label{ssec: Gluing the Fourier}

As we will be juggling different (co)-fiber sequences of sheaves of connective spectra,
let us recall some yoga about the standard $t$-structure on sheaves of spectra. 
We refer the reader 
to~\cite[\href{https://www.math.ias.edu/~lurie/papers/SAG-rootfile.pdf\#section.1.3.1}{Sec.~1.3.1}]{SAG} 
as well as~\cite{haine2025} for details.

For $\sigma\in \{\et,\Nis\}$, we from now on denote by
\[\Sp_{\sigma}(\base)\coloneq \Shv_{\sigma}(\Sm{\base};\Sp),
\quad \Spcn_{\sigma}(\base)\coloneq \Shv_{\sigma}(\Sm{\base};\Spcn)  \]
the category of $\sigma$-sheaves of (connective) spectra on $\Sm{\base}$.
By~\cite[\href{https://www.math.ias.edu/~lurie/papers/SAG-rootfile.pdf\#theorem.1.3.2.7}{Prop.~1.3.2.7}]{SAG} the natural $t$-structure on $\Sp$ induces a $t$-structure
on $\Sp_{\sigma}(\base)$ called the \tdef{standard $t$-structure}. Explicitly,
the fully faithful inclusion $i\colon \Spcn \subseteq \Sp$ induces an adjunction
\[ i^{\sigma}\colon \Spcn_{\sigma}(\base) \adj \Sp_{\sigma}(\base)\noloc\tau_{\geq 0}^{\sigma},\]
where $i^{\sigma}$ is again fully faithful and exhibits $\Spcn_{\sigma}(\base)$
as the full subcategory 
\[\Spcn_\sigma(\base) \simeq (\Sp_{\sigma}(\base))_{\geq 0} \subseteq \Sp_{\sigma}\]
of connective objects with respect to the standard $t$-structure.
Here, $\tau_{\geq 0}^{\sigma}$ is computed pointwise by taking connective covers and 
$i^{\sigma}$ is computed by applying $i$ pointwise and then sheafifying.

Note that the sheafification built into the functor $i^{\sigma}$ means that, if
$E\in \Sp_{\sigma}^{\cn}(\base)$ is a sheaf of connective spectra, the sheaf 
$i^{\sigma}E$ can sectionwise have negative homotopy groups. 
Indeed, these are the (positive) sheaf cohomology groups, so for any $U\in \Sm{\base}$ and
$E\in \Spcn_{\sigma}(\base)$ we write
\[ \mdef{H^n_{\sigma}(U;E)} \coloneq \pi_{-n}\Gamma(U;i^{\sigma}E).\]
This subtlety disappears if we sheafify again. For each $n\in \ZZ$, we have a functor
\[ \mdef{\pi^{\sigma}_n} \colon \Sp_{\sigma}(\base)\too \Sp^{\heart}_{\sigma}(\base)
\simeq \Shv_{\sigma}(\Sm{\base};\Ab),\]
which, for $E\in \Sp_{\sigma}(\base)$, is computed as the sheafification of the presheaf
$U\mapsto \pi_n\Gamma(U,E)$. Moreover, $E\in \Sp_\sigma(\base)$ is contained in the 
full subcategory $\Sp_\sigma(\base)_{\geq 0}$, i.e.~is given by the sheafification
of a sheaf of connective spectra, if and only if $\pi^{\sigma}_n E\simeq 0$ for 
all $n<0$. 

Before we begin properly, we make a remark about Weil restriction of commutative group
schemes.

\begin{rem}
Let $p\colon \base^\prime \to \base$ be a finite \'etale map. The pullback functor 
\[ p^{\ast}\colon \etSpcn(\base) \too \etSpcn(\base')  \]
admits both adjoints $p_{\sharp} \dashv p^{\ast} \dashv p_{\ast}$ which are naturally
equivalent by \cref{lem: fet semi-add}. In particular, 
on representable objects, namely commutative group schemes, we have a left adjoint functor
$p_{\sharp}=\Weilres{\base'}{\base}{-}$ which is underlying given by \emph{Weil restriction}
of schemes.
\end{rem}

\begin{construction}\label{cons: Weil resolution}
Let $\base$ be a scheme and $\t$ be a torus over $\base$. Given a finite \'etale map 
$p\colon \base^\prime \to \base$ which splits $\t$, let
$p_{\ast}p^{\ast}\t =\wt$ denote the Weil restriction of $\t_{\base^{\prime}}$ along $p$ and denote by $\kt \to \wt$ the kernel of the counit map $\wt \to \t$. 
Then $\kt$ is again a torus, and we obtain a short exact sequence
\begin{equation} \label{eq:exact_seq_tori}
    0 \too \kt \too \wt \too \t \too 0
\end{equation}
of algebraic tori over $\baf$, which we call the \tdef{Weil resolution} of $\t$ associated
to the cover $p\colon \base^{\prime}\to \base$. We sometimes leave the choice of a splitting
cover implicit and simply refer to a Weil resolution of $\t$. 
\end{construction}

Given a Weil-resolution as in~\eqref{eq:exact_seq_tori} upon taking character
lattices, see \cref{defn: char lattice}, and restricting to the smooth locus, we get a sequence of \'etale sheaves of abelian groups on $\Sch_{\base}$ of
the form
\begin{equation}\label{eq: ses}
 0 \to \cow{\t} \to \cow{\wt} \to \cow{\kt} \to 0\,,
\end{equation}
which is again exact, as can be seen by passing to the splitting extension 
$\base^{\prime}$.
Since \'etale covers of smooth schemes are smooth, we may argue as in the proof
of \cref{lem:C-fully-faithful} to see that the functor 
\[ \Shv_{\et}(\Sch_{\base};\Spcn)\too \Shv_{\et}(\Sm{\base};\Spcn) = \etSpcn(\base)  \]
induced by restricting along the inclusion $\Sm{\base}\subseteq \Sch_{\base}$ preserves
colimits. Combining this with the fact that the inclusion of the heart 
takes short exact sequences to cofiber sequences, we learn that we can rotate 
the sequence~\eqref{eq: ses} once to obtain a cofiber sequence in
$\etSpcn(\base)$ of the form
\begin{equation}\label{eq:exact_seq_char}
\cow{\wt} \too \cow{\kt} \too \Sigma\cow{\t}\,.
\end{equation}

Our first goal will be showing that this remains a cofiber sequence
after forgetting to Nisnevich sheaves of connective spectra. 

The inclusion $v\colon \Shv_{\et}(\Sm{\base};\Spc)\to \Shv_{\Nis}(\Sm{\base};\Spc)$ 
induces a commutative diagram of right adjoints
\[\begin{tikzcd}
	{\Sp_{\Nis}(\base)} & {\Sp_{\et}(\base)} \\
	{\Spcn_{\Nis}(\base)} & {\Spcn_{\et}(\base)}\,.
	\arrow["{\tau_{\geq 0}^{\Nis}}"', from=1-1, to=2-1]
	\arrow["{V_{\ast}}"', from=1-2, to=1-1]
	\arrow["{\tau_{\geq 0}^{\et}}", from=1-2, to=2-2]
	\arrow["{v_{\ast}}", from=2-2, to=2-1]
\end{tikzcd}\]
which, upon passing to left adjoints gives a commutative diagram 
\[\begin{tikzcd}
	{\Sp_{\Nis}(\base)} & {\Sp_{\et}(\base)} \\
	{\Spcn_{\Nis}(\base)} & {\Spcn_{\et}(\base)}\,,
	\arrow["{V^{\ast}}", from=1-1, to=1-2]
	\arrow["{i^{\Nis}}", from=2-1, to=1-1]
	\arrow["{v^{\ast}}"', from=2-1, to=2-2]
	\arrow["{i^{\et}}"', from=2-2, to=1-2]
\end{tikzcd}\]
Note that, while the functor $V_{\ast}$ preserves fiber and cofiber sequences, in general
it does not preserve the full subcategory of connective objects.

\begin{lem} \label{vanishing-pi-1}
Let $\wt= q_{\ast} \wt^{\prime}$ be the
Weil restriction of a split torus $\wt^{\prime}$ along a finite \'etale map 
$q\colon \base^\prime \to \base$. 
Then, we have an equivalence of sheaves
\[ \pi_{-1}^{\Nis} V_{\ast} \cow{\wt} \simeq 0 \in \Sp_{\Nis}^{\heart}(\base) \,. \] 
\end{lem}

\begin{proof}
By definition, $\pi_{-1}^{\Nis}(V_{\ast} \cow{\wt})$ is the Nisnevich sheafification
of the presheaf 
\[ \Sm{S}^{\op}\too \Ab, \quad  U\mapsto H^1_{\et}(U;\cow{\wt}) \,. \] 
We claim that this presheaf is already the zero sheaf.
Indeed, we have natural equivalences
\[
\cow{\wt} \simeq \map_{\ZZ_{\base}}(\wt,\GG_m) 
\simeq \map_{\ZZ_{\base}}(q_*\wt',\GG_m)   
\simeq  q_* \map_{\ZZ_{\base^{\prime}}}(\wt',q^*\GG_{m}) \simeq q_*\cow{\wt'}.
\] 
As a result, for any $U\in \Sm{\base}$ we get that 
\[
H^1_{\et}(U;\cow{\wt}) \simeq H^1_{\et}(U_{\base'};\cow{\wt'}) \,,
\]
where $\cow{\wt^{\prime}}\simeq \ZZ^{d}_{\base^{\prime}}$ is a constant sheaf. The right hand side is thus computed by the set of continuous
group homomorphisms
\[ H^1_{\et}(U_{\base'};\cow{\wt'}) \simeq
\rm{Hom}_{\rm{cont}}(\pi_{1}^{\et} U^\prime,\ZZ^{d}),\]
where $\ZZ^{d}$ carries the discrete topology. Since $\base=\Spec{F}$ is a field it is in
particular geometrically unibranch. Since $U$ is smooth over $\base$ and $\base^{\prime}$ 
is finite \'{e}tale over $\base$, $U'$ is also geometrically unibranch and thus 
$\pi_{1}^{\et}U'$ is a pro-finite group. In particular
there are no non-trivial continuous group homomorphisms $\pi_{1}^{\et} U' \to \ZZ^d$
and so $\pi_{-1}^{\Nis}V_{\ast} \cow{\wt}$ is the sheafification of the
zero presheaf, hence itself zero.
\end{proof}

From this ''acyclicity'' lemma about Weil restricted tori we learn the following.

\begin{prop}\label{cofiber sequence lat}
Let $\base$ be geometrically unibranch and $\t$ a torus over $\base$. 
    Given any finite \'etale extension $\base^{\prime}\to \base$ that splits $\t$,
    the sequence of connective Nisnevich spectra
    \[ v_{\ast} \cow{\wt} \too v_{\ast} \cow{\kt} \too v_{\ast} \Sigma \cow{\t}\]
    obtained by applying the forgetful functor $v_{\ast}\colon \etSpcn(\base)\to \nsSpcn(\base)$ 
    to the rotated Weil resolution~\eqref{eq:exact_seq_char}, is a cofiber sequence 
    in $\nsSpcn(\base)$.
\end{prop}
\begin{proof}
Since the functor $v_{\ast}$ is a right adjoint and hence preserves limits, the sequence in
the statement is a \emph{fiber} sequence. To show that it is a \emph{cofiber} sequence 
of sheaves of connective spectra, it suffices to show that boundary map 
\[ \pi_0^{\Nis}V_{\ast} \Sigma \cow{\t} \too \pi_{-1}^{\Nis} V_{\ast} \cow{\wt}
 \in \Sp_{\Nis}^{\heart}(\base)\]
 of the non-connective restrictions is the zero map. However, this is clear, since the target
 is $0$ by \Cref{vanishing-pi-1}.  
\end{proof}

In summary, we obtain a pushout square of motivic group rings.

\begin{cor}\label{topEM}
Let $\base$ be geometrically unibranch and $\t/\base$ be a torus.
Then for any Weil resolution of $\t$, the induced cofiber sequence~\eqref{eq:exact_seq_char}
induces a pushout square of motivic group rings
\[\begin{tikzcd}
\KGL_{\base}[\cow{\wt}] \ar[r] \ar[d] & \KGL_{\base} \ar[d]\\
\KGL_{\base} [\cow{\kt}] \ar[r] & \KGL_{\base}[ \Sigma \cow{\t}].
\end{tikzcd}\]
in the category of commutative $\KGL_{\base}$-algebras $\CAlg_{\KGL_{\base}}(\SH{\base})$. 
\end{cor}
\begin{proof}
The group ring functor $\SS[-]\colon \nsSpcn(\base) \to \CAlg(\SH{\base})$ is colimit
preserving by construction and so is the base change to $\KGL_{\base}$, so the claim follows from
\cref{cofiber sequence lat}.
\end{proof}

This describes the \emph{source} of the motivic Fourier transform 
$\Four_{\rm{mot}}^{\Sigma \cow{\t}}$ as a pushout of objects where we know 
$\Four_{\rm{mot}}$ to be an equivalence by \cref{thm: Fourier computation KGL}.
It remains to analyze the target.

Recall from \cref{defn: picard} that we defined the Cartier dual functor $(-)^{\vee}$
as the internal mapping object $\map_{\ZZ_{\base}}(-, \Sigma \Gm)$.
In particular $(-)^{\vee}$ takes cofiber sequences to fiber sequences.
Thus, applying $(-)^{\vee}$ to~\eqref{eq:exact_seq_char} and using the identifications
of \cref{lem: torus dual computation}, we get a pullback diagram on  
underlying stacks of the form
\begin{equation}\label{pullback tori}
   \begin{tikzcd}
       \t \ar[d] \ar[r] & \base \ar[d]\\
       \deloop{\kt} \ar[r] & \deloop{\wt}\, .
   \end{tikzcd}
\end{equation}

We claim that this pullback square induces a pushout in $\KGL_{\base}$-cohomology.

\begin{prop}[Eilenberg--Moore]\label{algEM}
    Let $\base=\Spec{\baf}$ be a field and $\t$ be a torus over $\base$.
    Given any Weil resolution
    \[ 0\to \kt \to \wt \to \t \to 0\]
    of $\t$, the associated pullback square \eqref{pullback tori} induces a pushout diagram 
    \[
    \begin{tikzcd}
       \KGL_{\base}^{\deloop{\wt}} \ar[r] \ar[d] & \KGL_{\base} \ar[d]\\
       \KGL_{\base}^{\deloop{\kt}} \ar[r] & \KGL_{\base}^{\t}.
    \end{tikzcd}
    \]
    in the category of commutative $\KGL_{\base}$-algebras $\CAlg_{\KGL_{\base}}(\SH{\base})$.
\end{prop}

The proof of \cref{algEM} will be the topic of \cref{ssec: EM}. Before we delve into this,
we observe that this immediately results in our main theorem.

\begin{thm}\label{thm: mainthm}
Let $F$ be a field and write $\base = \Spec{\baf}$. For any torus $\t$ with character
lattice $\cow{\t}$, the motivic Fourier transform 
of \cref{defn: mot fourier} gives a natural equivalence
\[ \Four^{\Sigma \cow{\t}}_\rm{mot} \colon \KGL_{\base}[\Sigma \cow{\t}]
\xrightarrow{\sim} \KGL_{\base}^{\t}\,,\]
of commutative $\KGL_{\base}$-algebras in $\SH{\base}$. 
\end{thm}
\begin{proof}
First note that by \cref{lem: torus dual computation} we have 
$(\Sigma \cow{\t})^{\vee}\simeq \t$, so the motivic Fourier transform
for $\Sigma\cow{\t}$ is defined and has target $\KGL_{\base}^{\t}$.
For any choice of splitting field $E/\baf$, the associated Weil resolution
combined with the naturality of the Fourier transform provides us with a commutative
diagram in the category $\CAlg_{\KGL}(\SH{\base})$ of the form:
\[\begin{tikzcd}
	&&& {\KGL_{\base}^{\deloop{\wt}}} && \KGL_{\base} \\
	{\KGL_{\base}[\cow{\wt}]} && \KGL_{\base} & {\KGL_{\base}^{\deloop{\kt}}} && {\KGL_{\base}^{\t}} \\
	{\KGL_{\base}[\cow{\kt}]} && {\KGL_{\base}[\Sigma\cow{\t}]}
	\arrow[from=1-4, to=1-6]
	\arrow[from=1-4, to=2-4]
	\arrow[from=1-6, to=2-6]
	\arrow["{\Four_{\rm{mot}}^{\cow{\wt}}}", from=2-1, to=1-4]
	\arrow[from=2-1, to=2-3]
	\arrow[from=2-1, to=3-1]
	\arrow["\id"{pos=0.3}, from=2-3, to=1-6]
	\arrow[from=2-3, to=3-3]
	\arrow[""{name=0, anchor=center, inner sep=0}, from=2-4, to=2-6]
	\arrow["{\Four_{\rm{mot}}^{\cow{\kt}}}"{pos=0.4}, from=3-1, to=2-4]
	\arrow[""{name=1, anchor=center, inner sep=0}, from=3-1, to=3-3]
	\arrow["{\Four_{\rm{mot}}^{\Sigma\cow{\t}}}"', from=3-3, to=2-6]
\end{tikzcd}\]
Now \cref{topEM} and \cref{algEM} tell us that the front and back planar squares are pushouts, respectively.
Moreover, by \cref{thm: Fourier computation KGL}, we know that the Fourier 
transforms $\Four_{\rm{mot}}^{\cow{\wt}}$ and $\Four_{\rm{mot}}^{\cow{\kt}}$ are
equivalences and so the lower right map 
\[\Four_{\rm{mot}}^{\Sigma\cow{\t}}\colon \KGL_{\base}[\Sigma\cow{\t}]\too \KGL_{\base}^{\t}\] 
is an equivalence as well, as claimed.
\end{proof}

Note that, for a split torus, this recovers the highly structured version of Quillen's
computation \cite[Theorem 8]{QuillenKTheory} of the algebraic $K$-groups of the scheme $\Gm$: 
Let $\baf$ be a field and $t\in \sO(\Gm)$ be a choice of coordinate Considering $t$ as an
automorphism of the trivial line bundle on $\Gm$, we learn that the associated map
   \[ t\colon S^1 \too \rm{QCoh}(\Gm)\] 
   induces an equivalence of commutative $K(\baf)$-algebras
    \[  K(\baf)[S^1]\xrightarrow{\sim} K(\Gm) \,.\]    
In \cref{sec: Kthy top mirror}, we leverage 
\cref{thm: mainthm} computationally by transferring our result into the world of 
equivariant homotopy theory and obtain a spectral sequence computing $K(\t)$.
Before we can reap those rewards, we must prove \cref{algEM}. To this end,
we need to introduce some further machinery.

\subsection{Artin--Tate motives}\label{ssec: EM}

The crucial technical input for \cref{algEM} will be that, for a Weil restricted
torus $\wt$, the zero section $\base \to \deloop{\wt}$ satisfies a strong
cellularity condition. Namely, it is \emph{Artin--Tate} in the sense
of~\cite[\href{https://arxiv.org/pdf/2010.10325v2\#nul.1.5}{Definition 1.5}]{artin_tate_recons}. We review the notions of Artin(--Tate) motives\footnote{Our definition is a small category version of the definition in \cite{artin_tate_recons}.}
and discuss in \Cref{EM-for-R-compactifiable} how, for an stack $p\colon \fX\to \base$ such that $p$ induces an equivalence
on $\pi_1^{\et}$, any pullback diagram of smooth representable stacks over $\fX$
gives an Eilenberg--Moore formula for $\KGL^{\fX}\in \SH{\base}$, if one of the legs is Artin--Tate.

Let $\fX$ be a stack and $R\in \CAlg(\SH{\fX})$ be a motivic commutative ring spectrum.
A \tdef{virtual vector bundle} on $\fX$ is a class $\sE\in \pi_0K(\fX)$.
Recall that we denote by 
\[ \mdef{R\twist{-}}\colon  \pi_0K(\fX)\too \Pic(\Mod{\SH{\fX}}{R}) \]
the motivic Thom-construction.

\begin{defn}\label{def:comp}
    Let $R \in \CAlg(\SH{\fX})$ a motivic commutative ring spectrum over $\fX$.
    We say that an object of $\Mod{\SH{\stack}}{R}$ is:
    \begin{enumerate}
        \item \tdef{Artin--Tate} if it contained in the thick subcategory generated by the objects of the form 
        \[\psi_\sharp ((\psi^\ast R)\twist{\sE}),\]
    where $\psi \colon \fY \to \fX$ is finite \'{e}tale and $\sE$ is a virtual
    vector bundle on $\fY$.
        \item \tdef{Artin} if it is contained in the thick subcategory generated by objects of the form 
        \[\psi_\sharp \psi^\ast R,\] 
        where $\psi \colon \fY \to \fX$ is finite étale.
    \end{enumerate}
    Moreover, we say that a smooth representable stack $\varphi\colon\fY \to \fX$ is 
    \tdef{Artin--Tate over $R$} (or \tdef{Artin over $R$}, respectively) if 
    $\varphi_\sharp \varphi^\ast R\in \Mod{\SH{\stack}}{R}$ is Artin--Tate
    (or Artin, respectively). If $R$ is the unit, we omit the phrase \enquote{over $R$}.
\end{defn}
\begin{rem}
Clearly, every Artin motive is Artin--Tate. Moreover, recall that, for a finite
\'etale map $\psi$, we have a natural equivalence
$\psi_\sharp \simeq \psi_*$, hence we can freely replace $\psi_\sharp$ by
$\psi_\ast$ when working with \Cref{def:comp}.
\end{rem}

Our goal is to show that if $W$ is the Weil restriction of a split torus along a finite 
\'etale map, then the $W$-equivariant motive of $W$ itself, i.e.~the smooth map of
stacks $\pt \too \deloop{W}$ is Artin--Tate and hence Artin over $\KGL_{\deloop{W}}$ 
by \cref{lem:Artin-Tate-implies-KGL-Artin}. To do this, let us first gather some general facts about Artin and Artin--Tate motives.

\begin{lem}\label{prop:Artin-Tate-are-dualizable}
All Artin--Tate motives are dualizable.
\end{lem}
\begin{proof}
Note that Thom spectra are invertible. Since dualizable objects form a 
thick subcategory closed under tensor products, it suffices to show that, for a
finite \'etale map $\psi$, the functor $\psi_{\sharp}$ preserves dualizable objects,
which is immediate from Atiyah duality \cref{thm: Atiyah Duality}.
\end{proof}

\begin{lem}\label{lem:artin-tate-base change}
    Let $f\colon R \to S$ be a morphism in $\CAlg(\SH{\fX})$. Then the base change along
    $f$ preserves the class of Artin--Tate and Artin objects.
\end{lem}
\begin{proof}
    Let $\phi \colon \fY \to \fX$ be finite \'etale and $\sE$ be a virtual vector
    bundle on $\fY$. 
    Then, by the projection formula for the adjunction $\phi_{\sharp}\dashv \phi^{\ast}$
    of \cref{thm: lower sharp proj formula}, we obtain an equivalence in
    $\Mod{\SH{\fX}}{S}$ of the form
\[
    S \otimes_R \phi_\sharp \phi^\ast R\twist{\sE} 
    \simeq \phi_\sharp (\phi^\ast S \otimes_{\phi^{\ast} R} \phi^\ast R\twist{\sE})
    \simeq \phi_\sharp \phi^\ast S \twist{\sE}
\]
    and so $S\otimes_R \phi_{\sharp} \phi^{\ast} R\twist{\sE}$ is Artin--Tate.
    Thus, the claim follows, as base change is exact and the respective subcategories
    are thick.
\end{proof}

\begin{lem}\label{lem:Artin-Tate-implies-KGL-Artin}
 Over the K-theory spectrum $\KGL$, a motive is Artin if and only if it is Artin--Tate.
 In particular, if $\fY\in \SmStk{\fX}$ is Artin--Tate, then it is Artin over $\KGL_{\fX}\in \SH{\fX}$.
\end{lem}
\begin{proof}
    By~\cite[\href{https://arxiv.org/pdf/2106.15001\#equation.10.7}{Thm.~10.7(ii)}]{kr24} we have, for any virtual vector bundle $\sE$ on $\fX$
    a canonical Bott-periodicity isomorphism
    \[\KGL_{\fX}\twist{\sE} \simeq \KGL_{\fX}\in \SH{\fX},\]
    which implies the first statement and the second is immediate by
    \cref{lem:artin-tate-base change}.
\end{proof}

\begin{lem}[Pullback]\label{lem:comp_pullback}
Let $\phi \colon \fX' \to \fX$ be a morphism of stacks. Then $\phi^*\colon \SH{\fX} \to \SH{\fX^\prime}$ preserves the subcategory of Artin--Tate objects. 
In particular, if $\fZ\in \SmStk{\fX}$ is Artin--Tate, then the pullback
$\phi^\ast \fZ\in \SmStk{\fX^\prime}$ is Artin--Tate.
\end{lem}

\begin{proof}
Let $\psi \colon \fY \to \fX$ be a finite \'{e}tale morphism $\sE$ be a virtual vector
bundle on $\fY$ and consider the pullback diagram
\[
    \begin{tikzcd}[sep=huge]
        \fY^\prime \ar[r,"{\phi'}"] \ar[d,"{\psi'}"]  & \fY\ar[d,"\psi"] \\
        \fX^\prime \ar[r,"\phi"]  & \fX.
    \end{tikzcd}
\]
Then, by $\sharp$-base change (\Cref{thm: lower sharp proj formula}
), we have equivalences
\[
\phi^*\psi_\sharp (\one_{\fY} \twist{\sE}) \simeq \psi'_\sharp {\phi'}^*( \one_{\fY}\twist{\sE}) \simeq \psi'_\sharp(\one_{\fY'} \twist{{\phi'}^* \sE}) 
\]
Since the right hand term is Artin--Tate by definition, the claim follows by virtue
of $\phi^\ast$ being exact and the subcategory of Artin--Tate motives being thick.
\end{proof}

\begin{lem}[Pushforward]\label{lem:comp_push}
Let $\phi \colon \fX' \to \fX$ be a finite \'{e}tale morphism of stacks. Then $\phi_\sharp \colon \SH{\fX^\prime} \to \SH{\fX}$ preserves the subcategory of Artin--Tate objects. 
\end{lem}

\begin{proof}
This follows immediately from the definitions, using that finite \'{e}tale morphisms are closed under composition.
\end{proof}

Unfortunately, being Artin--Tate is not closed under gluing of smooth $\fX$-stacks, but 
we have the following partial result which will suffice for our needs. 

\begin{lem}[Dévissage]\label{lem:comp_two_out_of_three}
    Suppose we have a diagram of smooth stacks
    \[ \begin{tikzcd}[sep=huge]
        \fZ \ar[r,"i"] \ar[d,swap,"\alpha"] & \fX \ar[d,"{\phi}"]  & \fU \ar[l,swap,"j"] \\ 
        \fY^\prime \ar[r,"{\beta}"] & \fY 
    \end{tikzcd} \]
    in which $\phi$ and $\alpha$ are smooth, $\beta$ is finite \'etale and $i$ is a closed
    immersion with complement $j$. Assume that the class of the normal bundle
    $[\sN_{\fZ/\fX}]$ is in the image of the map $\alpha^*\colon K_0(\fY') \to K_0(\fZ)$. Moreover, assume that $\alpha$ is Artin--Tate.
    Then, $\phi \circ j$ is Artin--Tate if and only if $\phi$ is Artin--Tate.
\end{lem}
\begin{proof}
By \cite[\href{https://arxiv.org/pdf/2106.15001\#equation.5.6}{Remark 5.6}]{kr24} we have a cofiber sequence 
  \[
  j_\sharp \one_{\fU} \too \one_{\fX}  \too i_*\one_{\fZ}. 
  \]
  Applying $\phi_\sharp$ we thus get a cofiber sequence 
  \[
  (\phi\circ j)_\sharp \one_\fU \too \phi_\sharp \one_\fX \too \phi_\sharp i_* \one_\fZ. 
  \]
  As the Artin--Tate subcategory is 
  closed under extensions and fibers, it is enough to show that the right hand term
  $\phi_\sharp i_* \one_\fZ$ is also Artin--Tate.
  Let $\sE \in K(\fY')$ be a preimage of $\sN_{\fX/\fZ}$.
  By \cite[\href{https://arxiv.org/pdf/2106.15001\#equation.6.6}{Theorem 6.6}]{kr24}, we have 
  \[
  \phi_\sharp i_*\one_\fZ \simeq (\phi\circ i)_\sharp (\one_\fZ \twist{\sN_{\fX/\fZ}}^{\vee}) \simeq \beta_\sharp \alpha_\sharp (\one_\fZ \twist{\sN_{\fX/\fZ}^{\vee}}) \simeq \beta_\sharp((\alpha_\sharp \one_\fZ) \twist{\sE^{\vee}}).
  \]
  Since $\alpha_\sharp \one_{\fZ} \in \SH{\fY'}$ is Artin--Tate, the Thom-spectrum
  $\alpha_\sharp \one_{\fZ} \twist{\sE^{\vee}}$ is as well.
  By \Cref{lem:comp_push} and the assumption that $\beta$ is finite \'{e}tale, we deduce
  that $\beta_\sharp(\alpha_\sharp \one_{\fZ} \twist{\sE^{\vee}}) \in \SH{\fY}$ is Artin--Tate
  and the result follows.
\end{proof}

\subsubsection{Eilenberg--Moore and Künneth}

We now deduce formulas for the cohomology of pullback square of stacks where 
one leg is Artin--Tate.

\begin{prop}[Künneth]\label{prop: Kunneth}
Let $\fX$ be a stack, $R\in \CAlg(\SH{\fX})$ and suppose we are given two 
smooth, representable maps $s\colon \fA\to \fX$ and $i\colon \fB \to \fX$. Denote by
$\delta\colon \fZ \to \fX$ their fiber product in $\Sm{\fX}$. The K\"unneth map 
   \[\ez_*\colon s_\ast s^\ast R \otimes_R i_\ast i^\ast R
   \too \delta_\ast \delta^\ast R ,\]
   is an equivalence if $s$ is Artin--Tate.
\end{prop}
\begin{proof}
Recall from the discussion of \cref{thm: lower sharp proj formula} that, for any
smooth representable map of stacks $f$, the functors $f_{\sharp}$ and $f_{\ast}$ are 
weakly dual to each other, that is we have
$(f_{\sharp}f^{\ast})^{\vee}\simeq f_{\ast}f^{\ast}$. Under
this identification, we may factor the map $\ez_{\ast}$ as the composite
\[ s_\ast s^\ast R \otimes_R i_\ast i^\ast R 
\simeq (s_\sharp s^\ast R)^{\vee} \otimes_R (i_\sharp i^\ast R)^{\vee}
\to (s_\sharp s^\ast R \otimes_R i_\sharp i^\ast R)^{\vee} 
\xrightarrow{\ez_{\sharp}^{\vee}} (\delta_{\sharp}\delta^{\ast}R)^{\vee}
\simeq \delta_{\ast}\delta^{\ast}R\]
where $\ez_{\sharp}$ is the natural map 
   \[\ez_\sharp \colon \delta_\sharp \delta^*R 
   \xrightarrow{\sim} s_\sharp s^*R\otimes_R i_\sharp i^*R,\]
which is an equivalence by the definition of the monoidal structure on 
$\Mod{\SH{\fX}}{R}$. Thus, $\ez_\ast$ is an equivalence if and only if
the map 
\[ (s_\sharp s^\ast R)^{\vee} \otimes_R (i_\sharp i^\ast R)^{\vee}
\to (s_\sharp s^\ast R \otimes_R i_\sharp i^\ast R)^{\vee} \]
is an equivalence. This holds because $s_{\sharp}s^{\ast}R$
is dualizable as an $R$-module by
\cref{prop:Artin-Tate-are-dualizable}.
\end{proof}

\begin{construction}
Let $p\colon \stack \to \base$ be a stack and $R\in \CAlg(\SH{\stack})$.
    For modules $M,N \in \Mod{\SH{\stack}}{R}$, the lax monoidality of $p_{\ast}$
    provides us with a comparison map of the relative tensor products
    \begin{equation}\label{eq: EM comparision map}
  \mdef{\tau_{M,N}}: p_\ast M \otimes_{p_\ast R} p_\ast N
  \too p_\ast \left( M \otimes_{R} N \right) \in \Mod{\SH{\base}}{p_\ast R} \,. 
    \end{equation}
  For fixed $N$, we denote by 
  \[\mdef{\cC(N,p)}\subseteq \Mod{\SH{\fX}}{R}\] 
  the full subcategory spanned by those modules $M$ for which $\tau_{M,N}$ is an equivalence.
\end{construction}

Note that, clearly we have $R\in \cC(N,p)$. Moreover, the category $\cC(N,p)$ is thick, 
as $p_\ast$ is exact. Moreover, the following proposition shows that transfers of $R$ 
along proper maps pulled back from the base, also lie in $\cC(N,p)$.

\begin{prop}\label{prop: em proper transfer}
    Let $R \in \CAlg(\SH{\stack})$ and $N \in \Mod{\SH{\stack}}{R}$.
    For a proper map of schemes $f \colon S^\prime \to S$, consider the pullback diagram
    \[
    \begin{tikzcd}[sep=huge]
        \stack^\prime \ar[d, "g"] \ar[r, "q"]  & \base^\prime \ar[d, "f"]\\
        \stack \ar[r, "p"] & \base .
    \end{tikzcd} 
    \]
     Suppose we have $M \in \Mod{\SH{\stack^\prime}}{g^\ast R}$, such that
     $M \in \cC(g^\ast N,q)$, then also $g_\ast M \in \cC(N,p)$. In particular, we have that $g_\ast g^\ast M \in \cC(N,p)$.
\end{prop}
\begin{proof}
    By assumption, we know that the natural map
    \begin{equation}\label{eq: proper em comp}
     f_\ast \left( q_\ast M \otimes_{q_\ast g^\ast R} q_\ast g^\ast N \right)
    \xrightarrow{\sim} f_\ast q_\ast \left( M \otimes_{g^\ast R} g^\ast N \right) 
    \end{equation}
    is an equivalence.
    Using the proper base change formula $q_\ast g^\ast \simeq f^\ast p_\ast$ and 
    then the proper projection formula for $f^\ast \dashv f_\ast$, we obtain a 
    natural equivalence between the left hand side and
    \[ f_\ast q_\ast M \otimes_{p_\ast R} p_\ast N 
    \simeq p_\ast g_\ast M \otimes_{p_\ast R} p_\ast N \,. \] 
    Moreover, using that $f_\ast q_\ast \simeq p_\ast g_\ast$ 
    and applying the proper projection formula for $g^\ast \dashv g_\ast$,
    we see that the right hand side is naturally equivalent to
    \[ p_\ast \left( g_\ast M \otimes_R N \right) \,. \]
    Under these identifications~\eqref{eq: proper em comp} thus becomes
    an equivalence
    \[ p_\ast g_\ast M \otimes_{p_\ast R} p_\ast N
    \xrightarrow{\sim} p_\ast \left( g_\ast M \otimes_R N \right),\]
    and we leave it to the reader to verify that this is indeed the natural map
    $\tau_{g_{\ast}M,N}$.
\end{proof}

We say that a stack $p\colon \stack\to \base$ is \tdef{\'etale simply connected}
if the map $p$ induces an equivalence on \'etale fundamental groups, i.e.~if every
finite \'etale cover $\stack$ is pulled back from $\base$.

\begin{cor}[Abstract Eilenberg--Moore]\label{EM-for-R-compactifiable}
    Let $p\colon \fX \to \base$ be an \'etale simply connected stack and let $R\in \CAlg(\SH{\fX})$.
     Then, if $A\in \Mod{\SH{\stack}}{R}$ is Artin, the natural transformation
     \[ p_\ast A \otimes_{p_\ast R} p_\ast(-)
  \too p_\ast \left( A \otimes_{R} - \right)\]
  is an equivalence.
\end{cor}
\begin{proof} 
It suffices to show that, for any $M\in \Mod{\SH{\fX}}{R}$ we have $A\in \cC(M,p)$.
Since $\cC(M,p)$ is thick, and $A$ is Artin, it suffices to prove this for 
$A=g_{\sharp}g^{\ast}R$ where $g\colon \stack^{\prime}\to \stack$ is any finite \'etale map.
By assumption on $\stack$, $g$ is pulled back from $\base$, so the claim
follows from \cref{prop: em proper transfer}.
\end{proof}

\subsection{Eilenberg--Moore for algebraic tori}\label{ssec: EM for tori}

We finally come to the proof of \cref{algEM} and begin with showing that the
zero section $\base\to \deloop{\wt}$ of a Weil restricted torus $\wt$ is Artin--Tate. The cell decomposition appearing in the proof was already used by Merkurjev and Panin in their computation of $K_0(\t)$ of an algebraic torus $\t$ \cite{mp97}.

Let $\baf$ be a field and let $G$ be its absolute Galois group.
For a finite set $I$ with an action on $G$ and a variety $X$ over $\baf$, let $X^I$ be
the Weil restriction of $X$ along the map $I\to \pt$, regarded as morphism of finite $\baf$-varieties. 
In other words,
\[
  X^I\times_{\baf}{\cl{\baf}} = (X\times_\baf \cl{\baf})^I
\]
with the action of the Galois group given by
\[
  \gamma(x)_i = \gamma(x_{\gamma^{-1}(i)}) \quad \forall x = (x_i)_{i\in I} \in (X\times_\baf \cl{\baf})^I.
\]

\begin{prop}\label{prop:point_to_BG_comp}
Let $\baf$ be a field with Galois group $G$ and let $I$ be a finite $G$-set. 
Then, the motive of the canonical smooth section
$e\colon \Spec{\baf} \to  \deloop{\GG_m^I}$ is Artin--Tate in $\SH{\deloop{\Gm^I}}$.
\end{prop}

\begin{proof}
We will prove the result simultaneously over all fields and by induction on $|I|$.  
For $I = \varnothing$ the map $e$ is the identity and the claim is trivially true. 

Now assume we have proven the claim for all fields $F$ and all finite $G$-sets
of cardinality $<n$ and let $I$ be a finite $G$-set with $|I|=n$.
Consider the $\GG_m^I$-equivariant embedding $\GG_m^I \into (\AA^1_F)^I\simeq \AA^{|I|}_F$. 
For a subset $J\subseteq I$ we let
  \[ U_{J} \subseteq (\AA^1_F)^I \]
  be the $\GG_{m,F}^I$-invariant $\cl{\baf}$-subvariety consisting of tuples
  $(a_i)_{i\in I} \in (\AA^1_F)^I$ for which $a_i \ne 0$ for $i\in J$. Moreover, for
  any $0 \le k \le |I|$, set
  \[
  U_k := \bigcup_{J\subseteq I,|J|=k} U_J \subseteq (\AA^1_F)^I, 
  \]
  that is, $U_k$ is the open subvariety consisting of tuples with at least $k$
  non-zero entries. 

  We will prove by another induction on $k$ that the smooth $B\GG_{m,F}^I$-stack
  $U_k//\GG_{m,F}^I$ is Artin--Tate, which yields our claim for $k=n=|I|$, as we have a canonical
  $\GG_{m,F}^{I}$-equivariant identification $U_{|I|}= \Gm^I$.
  For $k=0$, $U_0 = (\AA^1_F)^I$ is an affine space, so that its motive agrees with the unit. 

  For $k > 0$, consider the $\GG_{m,F}^I$-equivariant embedding $U_{k} \into U_{k-1}$, and denote
  by $Z_k$ the complement.
The group $G$ acts on the set $\binom{I}{k-1}$ of subsets of size $k-1$ of $I$.
Let $J_1,\dots,J_\ell$ be a complete set of representatives for the $G$-orbits on
$\binom{I}{k-1}$ and let $G_t$ be the stabilizer of $J_t$.
Let $\baf_t=\bar \baf^{G_t}$ be the fixed field of $G_t$. 
Then, we have an open and closed subvariety $Z_{J_t}\subseteq Z_k \times_{F} F_t$ consisting of tuples $(a_i)_{i\in I}$ such that $a_i \ne 0 \iff i\in J_t$. It fits into a diagram
\[ \begin{tikzcd} \coprod_{t=1}^l Z_{J_t}//\GG_{m,F_t}^I \ar[r] \ar[d] & U_{k-1}//\GG_{m,F}^I  \ar[d] & U_{k}//\GG_{m,F}^I \ar[l] \\
\coprod_{t=1}^l B\GG_{m,F_t}^I\ar[r] & B\GG_{m,F}^I \end{tikzcd} \] in which the upper row is an smooth open-closed decomposition and the lower horizontal arrow is finite étale.
Using the induction hypothesis for $k$ and \Cref{lem:comp_two_out_of_three}, we see that $U_k//\GG_{m,F}^I \to B\GG_{m,F}^I$ is Artin--Tate if 
the following hold:
\begin{enumerate}
\item The map $ Z_{J_t} //\GG_{m,F_t}^I  \to B\GG_{m,F_t}^I$ is Artin--Tate. 
\item The normal bundle of the map 
$Z_{J_t} // \GG_{m,F_t}^I \to U_{k-1} // \GG_{m,F}^I$ is pulled back from 
$ B\GG_{m,F_t}^I$.  
\end{enumerate}

For $(1)$, note that $Z_{J_t} \cong \GG_{m,F_t}^{J_t} \times (\AA_{F_t})^{|I \setminus J_t|}$ is acted on by
$\GG_{m,F_t}^I = \GG_{m,F_t}^{J_t} \times \GG_{m,F_t}^{I\setminus J_t}$ factor-wise, i.e. $\GG_{m,F_t}^{J_t}$ acts on the first factor freely and $\GG_{m,F_t}^{I \setminus J_t}$ acts on $(\AA_{F_t})^{|I \setminus J_t|}$. By $\AA^1$-invariance, the motive is the same as just the first factor, i.e. the motive of $\GG_{m,F_t}^{J_t} // \GG_{m,F_t} \to B\GG_{m,F_t}^I$
This morphisms of schemes sits in a pullback
\[
\begin{tikzcd}
 \GG_{m,F_t}^{J_t}//\GG_{m,F_t}^I \ar[r] \ar[d]& B\GG_{m,F_t}^I \ar[d] \\
 \Spec{F_t}\ar[r]                 & B\GG_{m,F_t}^{J_t}  
\end{tikzcd}
\]
The lower horizontal map is Artin--Tate by the inductive hypothesis, as $|J_t|=k-1<n=|I|$. 
Therefore, so is the upper horizontal map by \Cref{lem:comp_pullback}. 

For $(2)$, the normal bundle in question is the restriction of scalars of the (equivariant) normal bundle of $Z_{J_t}$ inside $U_{k-1} \times_{F} F_t$. 
Then, since $Z_{J_t}$ is an open subvariety of $\Aff_{F_t}^{J_t}$ which embeds linearly and equivariantly in $\Aff_{F_t}^{I}$, the normal bundle of $Z_{J_t}$ identifies with the pullback of the algebraic representation $F_t^{I\setminus J_t}$ of $\GG_{m,F_t}^{I\setminus J_t}$, along the map $Z_{J_t}//\GG_{m,F_t}^I \to B\GG_{m,F_t}^I \to B\GG_{m,F_t}^{I\setminus J_t}$. This yields the result.
\end{proof}

\begin{cor}\label{cor: torus is artin}
    Let $\baf$ be a field, write $\base = \Spec{\baf}$ and suppose $\wt$ is the Weil restriction
    of a split torus along a finite \'etale extension of $\base$. Then, the section
    $s\colon \base \too \deloop{\wt}$ is Artin over $\KGL_{\deloop{\wt}}$.
\end{cor}
\begin{proof}
    By \cref{prop:point_to_BG_comp} the map is Artin--Tate and hence Artin 
    over $\KGL_{\deloop{\wt}}$ by \cref{lem:Artin-Tate-implies-KGL-Artin}.
\end{proof}
Finally, we need the following lemma which will be used in order to verify the
condition appearing in \Cref{EM-for-R-compactifiable}.

\begin{lem}\label{lem: pi1 of BT}
    For any torus $\t$ over $S=\Spec{\baf}$, the stack $\deloop{\t}$ is \'etale simply
    connected relative to $S$, namely the structure map $p\colon \deloop{\t} \to S$ induces an isomorphism of the \'{e}tale fundamental groups.
\end{lem}
\begin{proof}
For any choice of separably closure $\baf \to E$ we have a short
exact sequence of \'etale fundamental groups
\[1 \to  \pi_1^{\et}(\deloop{\t}_{E})\to \pi_1^{\et}(\deloop{\t})
\to \mathrm{Gal}(E/\baf) \to 1\,. \]
As $\t$ is a connected group scheme and the delooping is in the \'etale topology,
the left hand term vanishes and the claim follows.
\end{proof}

With all this in hand, the proof of the desired Eilenberg--Moore formula is straightforward.

\begin{proof}[Proof of \Cref{algEM}]
Let $q\colon \deloop{\kt}\to \base $
and $p\colon \deloop{\wt} \to \base$ denote the unique maps to the terminal object. For the pullback diagram 
of stacks over $\base$
\[\begin{tikzcd}
	T & \base \\
	{\deloop{\kt}} & {\deloop{\wt}} \\
	& \base
	\arrow[from=1-1, to=1-2]
	\arrow[from=1-1, to=2-1]
	\arrow["\lrcorner"{anchor=center, pos=0.125}, draw=none, from=1-1, to=2-2]
	\arrow[from=1-2, to=2-2]
	\arrow[from=2-1, to=2-2]
	\arrow["q"', from=2-1, to=3-2]
	\arrow["p", from=2-2, to=3-2]
\end{tikzcd}\]
we want to show that the induced comparison map
\begin{equation}\label{eq: Final EM comparsion map}
       q_{\ast}\KGL_{\deloop{\kt}} \otimes_{p_{\ast}\KGL_{\deloop{\wt}}} \KGL
      \too \KGL^{\t} \in \CAlg_{\KGL}(\SH{\base}).
\end{equation}
is an equivalence. First, note that the left hand side is given by
\[  q_\ast \KGL_{\deloop{\kt}} \otimes_{p_{\ast} \KGL_{\deloop{\wt}}} \KGL
    \simeq p_\ast \KGL_{\deloop{\wt}}^{\deloop{\kt}} 
    \otimes_{p_\ast \KGL_{\deloop{\wt}}} p_\ast \KGL_{\deloop{\wt}}^{\base},\]
and similarly for the right hand side we have
 \[ \KGL^{\t} \simeq p_\ast \KGL_{\deloop{\wt}}^{\t}.\]
Under these identifications, the map~\eqref{eq: Final EM comparsion map}
can be written as a composite 
\begin{equation*}
 p_\ast \KGL_{\deloop{\wt}}^{\deloop{\kt}} \otimes_{p_\ast \KGL_{\deloop{\wt}}}
 p_\ast \KGL_{\deloop{\wt}}^{\base}
    \xrightarrow{\tau} p_\ast \left( \KGL_{\deloop{\wt}}^{\deloop{\kt}}
    \otimes_{\KGL_{\deloop{\wt}}} \KGL_{\deloop{\wt}}^{\base} \right)
    \xrightarrow{\mu}  p_\ast (\KGL^{\t}_{\deloop{\wt}})\in \SH{\base}.
\end{equation*}
As $\wt$ is the Weil restriction of a split torus, we know by 
\cref{cor: torus is artin} that the motive
$\KGL^{\base}_{\deloop{\wt}}\in \Mod{\SH{\deloop{\wt}}}{\KGL_{\deloop{\wt}}}$
is Artin, and so the map $\tau$ is an equivalence by
\cref{EM-for-R-compactifiable} and \cref{lem: pi1 of BT}. 
Moreover, the morphism $\mu$ is precisely the image under $p_\ast$ of the K\"unneth map 
\[\ez_{\ast}\colon \KGL_{\deloop{\wt}}^{\deloop{\kt}} 
\otimes_{\KGL_{\deloop{\wt}}} \KGL_{\deloop{\wt}}^{\base}
\too\KGL^{\t}_{\deloop{\wt}} \in \SH{\deloop{\wt}}\]
and thus an equivalence by \cref{prop: Kunneth}, so we are done.
\end{proof}

\section{K-Theory of tori via the topological mirror}\label{sec: Kthy top mirror}

Let $\t$ be an algebraic torus over a field $\baf$.
We want to leverage \cref{thm: mainthm} to actually give a presentation of $K$-theory spectrum $K(\t)$
in terms of the spectra $K(L)$ where $\baf \to L \to E$ runs over intermediate extensions.
This is done via the motivic--to--equivariant comparison functors which we reviewed
in \cref{ssec: motivic-vs-equivariant}. Let $\GalG$ denote the absolute Galois group of $\baf$.

We start in \cref{ssec: top tori} by showing that any topological torus with an action of $G$ is contained in the subcategory of Borel $\GalG$-spaces and use this to give a topological description of the delooped character lattice $\Sigma \Lambda^\ast(\t)$ in \Cref{ssec: homotopy type of BLambda}

With this in hand, we deduce \Cref{thm: thm A intro} as 
\Cref{cor: K theory as fixed points}, by applying the comparison functor 
$\SH{\baf}\to \Sp_{\GalG}$ to the equivalence of \cref{thm: mainthm} and showing that a
certainly assembly map is an equivalence of $\GalG$-spectra.

In \cref{ssec: G computation}, we record the Atiyah--Hirzebruch spectral
sequence associated to the equivalence of \cref{cor: K theory as fixed points}, 
which gives a tool for actually computing the $K$-\emph{groups} of $\t$. 
As a special case, we recover the formula for $K_0(\t)$ of Merkujev--Panin via the
generic formula for Bredon homology as a coend.
Finally, we illustrate how our formula recovers a computation of the $K$-theory
of algebraic tori of rank $1$ due to Swan.

\subsection{Topological \texorpdfstring{$\GalG$}{G}-tori} \label{ssec: top tori}

The goal of this section is to explain how to associate a topological torus with
$\GalG$-action to a lattice with a $\GalG$-action and to prove that the homotopy
type of this torus is Borel. 

\begin{construction}
Let $\mdef{\TopTor}$ denote the topological category of \tdef{topological tori},
i.e.~the category of commutative, connected, compact Lie groups. Moreover, let 
$\mdef{\Lattices}$ denote the category of \tdef{lattices}, i.e.~the full subcategory
of $\Ab$ spanned by the finitely generated free abelian groups. For a lattice
$\Lambda \in \Lattices$, write $\Lambda_{\RR}\coloneq \Lambda \otimes_{\ZZ}\RR$ 
for the base change considered as a topological space with discrete subspace
$\Lambda \subseteq \Lambda_{\RR}$. Then there is an equivalence of categories 
\[ \Lattices \xrightarrow{\sim} \TopTor \] 
given by sending a lattice $\Lambda$ to the topological torus $\Lambda_\R / \Lambda$ 
and inverse given by taking the first homology group $H_1(-,\mathbb{\Z})$. 
In particular, for any profinite group $\GalG$, we get an induced equivalence of $G$-equivariant categories and denote by
\[ \mdef{\mathbb{T}(-)}\colon \Lattices^{\deloop{G}}
\xrightarrow{\sim} \TopTor^{\deloop{G}} \,. \]
Here we adopt our convention from \Cref{ssec: motivic-vs-equivariant}, defining the category of objects with an action of a pro-finite group as the limit of the categories of its finite quotients.
\end{construction}

\begin{prop} \label{prop:torus-is-Borel}
    Let $\Lambda \in \Lattices^{\deloop{\GalG}}$ be a lattice with $\GalG$-action for $G$ a finite group.
    Then, the associated torus $\mirr{\Lambda}$ is Borel, i.e.~the natural map 
    \[ {\mirr{\Lambda}}^G \longrightarrow {\mirr{\Lambda}}^{hG} \] 
    from the fixed points to the homotopy fixed points is an equivalence.
\end{prop}
\begin{proof} 
    Both spaces in question are $1$-truncated, because ${\mirr{\Lambda}}$ is so. 
    Moreover, as ${\mirr{\Lambda}}$ is a group-like commutative monoid, 
    so are both fixed points, and it suffices to show that the map between them
    is an isomorphism on $\pi_0$ and $\pi_1$ at the canonical base point. 
    Consider the exact sequence 
    \[ 0 \to \Lambda \to \Lambda_\RR \to {\mirr{\Lambda}} \to 0.\] 
    Since $H^1(G;\Lambda_\RR) = 0$ we obtain an exact sequence of abelian Lie groups 
    \[ 0\to \Lambda^G \to \Lambda_\RR^G \to {\mirr{\Lambda}}^G 
    \to H^1(G;\Lambda) \to 0, \]
    or equivalently, a short exact sequence
    \[ 0\to \mathbb{T}(\Lambda^G) \to {\mirr{\Lambda}}^G \to H^1(G;\Lambda) \to 0.\]
    The Lie group on the right is a finite discrete group and the Lie group on the
    left is connected. It follows that  
    \[ \pi_0{\mirr{\Lambda}}^G \iso H^1(G;\Lambda)
    \simeq \pi_0(B\Lambda)^{hG} = \pi_0T(\Lambda)^{hG}\]
    while
    \[ \pi_1{\mirr{\Lambda}}^G \simeq \pi_1\mathbb{T}(\Lambda^G)\simeq \Lambda^G
    \simeq \pi_1(B\Lambda)^{hG} = \pi_1{\mirr{\Lambda}}^{hG}.\]
    It is now easy to check that these two equivalences are the ones induced from
    the map ${\mirr{\Lambda}}^G\to {\mirr{\Lambda}}^{hG}$, yielding the result. 
\end{proof}

We can summarize this result as follows: 
 \begin{cor}\label{comparison topological tori}
     For any profinite group $\GalG$, the following diagram commutes:
     \[ \begin{tikzcd}
        \Lattices^{\deloop{G}} \ar[r] \ar[d,"\mathbb{T}(-)"] & \Spcn({\deloop{G}}) \ar[r,"\Sigma"] & \Spcn({\deloop{G}}) \ar[d,"\beta"] \\
       \TopTor^{\deloop{G}} \ar[rr] & & \Spcn(\cO_{\GalG}^{\op})
   \end{tikzcd}\]
    where the unlabeled horizontal morphisms are forgetful functors, only remembering the underlying (Borel) $G$-space together with its multiplicative structure.
    \end{cor}
\begin{proof}
    As all categories in question are formally defined as limits of the finite quotients
    of $G$, we may assume that $G$ is finite. 
    We have to verify that the underlying homotopy type of $\mirr{\Lambda}$ is Borel, as a connective spectrum. As limits of connective spectra are computed on underlying spaces, it is enough to verify that it is a Borel space. This precisely is the content of \Cref{prop:torus-is-Borel}. 
\end{proof}

\subsection{Deducing the \texorpdfstring{$\GalG$}{G}-equivariant description}\label{ssec: homotopy type of BLambda}
Throughout this section, we let $\baf$ be a field with absolute Galois-group 
$\GalG$. 

Recall from the discussion in \cref{ssec: motivic-vs-equivariant} that
we have adjunctions
\[ C_{\et} \colon \Spcn(\deloop{\GalG}) \rightleftarrows  \Spcn_{\et}(\baf) \noloc U_{\et}\]
\[ C_{\Nis} \colon \Spcn( \cO_{\GalG}^{\op}) \rightleftarrows \Spcn_{\Nis}(\baf) \noloc U_{\Nis}\,,\]
such that the right adjoints commute with the natural inclusions
$\beta \colon \Spcn(\deloop{G})\to \Spcn(\cO_{\GalG}^{\op})$ and
$\nu_{\ast}\colon \etSpcn(\baf) \to \nsSpcn(\baf)$. 
For a torus $\t$, the considerations of \cref{ssec: top tori} allow us to 
give a more explicit description of the $\GalG$-space
$U_{\Nis}(\Sigma{\cow{\t}})$ as follows.

\begin{cor}\label{cor: top tori compact}
Let $\t$ be a torus over $\baf$ and let 
$\Lambda = U_{\et}\cow{\t} \in \Lattices^{\deloop{\GalG}}$ denote the underlying
lattice with $\GalG$-action of $\cow{\t}$. Then, we have a natural equivalence of commutative groups in genuine $G$-spaces:
\[U_{\Nis}(\Sigma{\cow{\t}}) \simeq \mirr{\Lambda}  \]
\end{cor}
\begin{proof}
    We know from \Cref{lem:C-fully-faithful} that we have
    \[ U_{\Nis}(\Sigma{\cow{\t}}) = \beta U_{\et} (\Sigma{\cow{\t}}) \simeq \beta \Sigma U_{\et} (\cow{\t}) \,. \]
    The right-hand side is by definition the Borel $G$-space associated to
    $\Sigma U_{\et} (\cow{\t})$ which is precisely is the delooping of the character lattice and agrees with
    $\mirr{\Lambda}$ by \Cref{comparison topological tori}, so we conclude.
\end{proof}

Put differently, all ways to associate a $\GalG$-equivariant topological torus 
to an algebraic torus $\t$ agree. Accordingly, this construction gets a name.

\begin{defn}
Let $\t/\baf$ be a n algebraic torus and let $\Lambda$ be the underlying lattice with 
$\GalG$-action of $\cow{\t}$. We call
\[\mdef{{\topTor{\t}}}\coloneq \mirr{\Lambda} \]
the \tdef{topological mirror} of $\t$. Note that, by \cref{cor: top tori compact}, 
the underlying $G$-space of $\topTor{\t}$ is \emph{compact} in $\Spc_{\GalG}$.
\end{defn}

Recall from \cref{prop: c-commutes-with-suspension} that the adjunction 
$C_{\Nis}\dashv U_{\Nis}$ on the level of sheaves of spectra further refines to an adjunction 
\[ \cC\colon \GSp{\GalG} \adj \SH{\baf}\colon \cU,\]
where $\cC$ is symmetric monoidal and $\cU$ is lax monoidal.

Our goal is to use this to describe the algebraic $K$-theory of $\t$ in terms 
of the equivariant homology of the mirror ${\topTor{\t}}$. 
To this end, let us first observe the following general fact.

\begin{lem} \label{lem: projection-formula-c-U}
    Let $X\in \GSp{\GalG}$ be dualizable and $E \in \SH{\baf}$ be arbitrary. 
    The natural map 
    \[ \mathcal U(E) \otimes X \too \mathcal U(E \otimes \cC(X)) \] 
    is an equivalence of $\GalG$-spectra.
\end{lem}
\begin{proof}
   This follows by general nonsense about monoidal adjunctions. Indeed, since $\cC$
   is symmetric monoidal, we have for any $A\in \GSp{\GalG}$ natural equivalences
   \begin{align*}
       \Map_{\GSp{\GalG}}(A, \mathcal{U}(E) \otimes X) 
       &\simeq \Map_{\GSp{\GalG}}(A \otimes X^{\vee}, \mathcal{U}(E))\\ 
       &\simeq \Map_{\SH{\baf}}(\cC(A) \otimes \cC(X)^{\vee}, E)\\
       &\simeq \Map_{\SH{\baf}}(\cC(A), E\otimes \cC(X))\\
       &\simeq \Map_{\GSp{\GalG}}(A, \mathcal{U}(E\otimes \cC(X))) 
   \end{align*}
   and so the claim follows by Yoneda.
\end{proof}

We want to apply this projection formula in the case where $X$ is the $G$-equivariant
suspension spectrum $\SS_{\GalG}[{\topTor{\t}}]= \SS_{\GalG}[U_{\Nis}(\Sigma\cow{\t})]$. To do this, we compute the following.

\begin{prop}\label{prop: Blattice-is-topological} 
    Let $\t$ be an algebraic torus over $\baf$ with character lattice $\cow{\t}$.
    The natural comparison map 
    \[ \cC (\SS_{\GalG}[U_{\Nis}(\Sigma \cow{\t})])
    \too \SS_{\baf}[\Sigma \cow{\t}] \in \SH{\baf}\]
    is an equivalence of motivic spectra.
\end{prop}
\begin{proof}
The map in question factors as the composite 
\[ \cC( \SS_{\GalG} [U_{\Nis}(\Sigma\cow{\t})]) 
\xrightarrow{\sim} \SS_{\baf}[C_{\Nis}U_{\Nis}(\Sigma \cow{\t})] \to \SS_{\baf}[\Sigma \cow{\t}],\]
where the left hand morphism is an equivalence by \cref{prop: c-commutes-with-suspension}. For the 
right hand side, we claim that the unit map
    \[ \left( C_{\Nis} \circ  U_{\Nis} \right) (\Sigma \Lambda^\ast (\t)) \too \Sigma \Lambda^\ast(\t) \] 
is an equivalence of Nisnevich sheaves valued in connective spectra. 
Consider the resolution of $\t$ via tori
\[ \kt \to \wt \to \t\]
as in~\eqref{eq:exact_seq_tori}. Then, by \cref{cofiber sequence lat}, we have a cofiber
sequence of Nisnevich connective spectra
\[ \cow{\wt} \to \cow{\kt} \to \Sigma \cow{\t}.\]
Since the composite $C_{\Nis} \circ {U}_{\Nis}$ preserves colimits by 
\cref{lem:C-fully-faithful} we obtain a map of cofiber sequences connective Nisnevich spectra
of the form
\[\begin{tikzcd}
	{\cow{\wt}} & {\cow{\kt}} & {\Sigma \cow{\t}} \\
	{(C_{\Nis} \circ {U}_{\Nis})(\cow{\wt})} & {(C_{\Nis} \circ {U}_{\Nis})(\cow{\kt})} & {(C_{\Nis} \circ {U}_{\Nis})(\Sigma\cow{\t})}\, .
	\arrow[from=1-1, to=1-2]
	\arrow[from=1-2, to=1-3]
	\arrow[from=2-1, to=1-1]
	\arrow[from=2-1, to=2-2]
	\arrow[from=2-2, to=1-2]
	\arrow[from=2-2, to=2-3]
	\arrow[from=2-3, to=1-3]
\end{tikzcd}\]
It follows from \Cref{lem:lattice-is-finite-etale} (together with \Cref{lem:C-fully-faithful} that the 
first two vertical maps are equivalences and so the right hand map is an equivalence
as well.
\end{proof}

This allows us to use \cref{lem: projection-formula-c-U} to import \cref{thm: mainthm}.
Recall from \cref{defn: G equiv K theory} that for any $X\in \Sm{\baf}$, the 
$\GalG$-equivariant algebraic $K$-theory spectrum of $X$ is defined as 
$K_{\GalG}(X)= \cU (\KGL_{\baf}^{X}) \in \CAlg(\GSp{\GalG})$.
As usual, we denote by $K_{\GalG}(\baf)[-]= \SS[-] \otimes K_{\GalG}(\baf)$ the 
symmetric monoidal enhancement of the homology functor. 

\begin{thm}\label{equivariant mainthm}
    Let $\t$ be an algebraic torus defined over a field $F$ with absolute Galois
    group $\GalG$. There is a natural equivalence of commutative $K_{\GalG}(\baf)$-algebras
    \[ K_{\GalG}(\baf)[{\topTor{\t}}] \xrightarrow{\sim} K_{\GalG}(\t) 
    \in \CAlg_{K_\GalG(\baf)}(\GSp{\GalG}) \]
    between the equivariant homology of the topological mirror $\topTor{\t}$ with coefficients in $K_{\GalG}(\baf)$ and the $\GalG$-equivariant algebraic $K$-theory of $\t$.
 \end{thm}
\begin{proof}
Applying the lax monoidal functor $\mathcal U$ to the map of \Cref{thm: mainthm}, yields
\[ \mathcal U(\Four^{\Sigma \cow{\t}}) \colon \mathcal U( \KGL[\Sigma\cow{\t}]) 
\xrightarrow{\sim} \mathcal U(\KGL^{\t})=K_{\GalG}(\t) 
\in \CAlg_{K_{\GalG}(\baf)}(\GSp{\GalG})\]
and so it suffices to show that we have a natural equivalence
\begin{equation}\label{eq: G ass}
 K_{\GalG}(F)[{\topTor{\t}}] \simeq \mathcal U( \KGL[\Sigma\cow{\t}])\,.
\end{equation}
Recall that ${\topTor{\t}}\simeq U_{\Nis}(\Sigma\cow{\t}) \in \Spc_{\GalG}$ is Borel
and compact by \Cref{cor: top tori compact}, which implies that the spectrum 
$\SS[U_{\Nis}(\Sigma \cow{\t})]\in \GSp{\GalG}$ is dualizable.
Thus, we may use the projection formula of \cref{lem: projection-formula-c-U} to see that 
the left hand side of~\eqref{eq: G ass} is given by
\[ K_{\GalG}(F)[{\topTor{\t}}]= \cU(\KGL) \otimes \SS[U_{\Nis}(\Sigma\cow{\t})]
\simeq \cU (\KGL \otimes \cC(\SS[U_{\Nis}(\Sigma\cow{\t})])).\]
Finally, by \cref{prop: Blattice-is-topological}, we have a natural equivalence
\[ \cU (\KGL \otimes \cC(\SS[U_{\Nis}(\Sigma\cow{\t})]))
\simeq \cU(\KGL \otimes \SS[\Sigma \cow{\t}])= \cU(\KGL[\Sigma \cow{\t}])\]
so we conclude.
\end{proof}

In particular, we immediately get the following by taking $\GalG$-fixed points.

\begin{cor}\label{cor: K theory as fixed points}
Let $\t/\baf$ be a torus with topological mirror ${\topTor{\t}}$. 
The algebraic $K$-theory spectrum of $\t$ is computed by the genuine fixed points
\[ C_{\ast}^{G}(\topTor{\t};K_{\GalG}(\baf))=\left( K_{\GalG}(\baf) \otimes {\topTor{\t}} \right)^{\GalG} 
\xrightarrow{\sim} K(\t) \in \CAlg(\Sp).\]
\end{cor}

\begin{rem}
    Let $K$ be a genuine $G$-spectrum and $X$ a genuine $G/N$-spectrum. Writing $\mathrm{infl}_{G/N}^{G}$ for the inflation functor from genuine $G/N$-spectra to genuine $G$-spectra, there is a natural equivalence of genuine $G/N$-spectra:
    \[ K^N \otimes X \simeq (K \otimes \mathrm{infl}^{G}_{G/N} X)^N.\] 

    Choose a splitting field $E$ of the torus $\t$, corresponding to a finite index normal subgroup $N<G$. 
    Then $\Lambda(\t)$ will be inflated from $G/N$ and so will be $\topTor{\t}$. 
    Applying $N$-fixed points to \Cref{equivariant mainthm}, we can use the above formula $K=K_{\GalG}$ and $X$ the $G/N$-topological mirror of~$\t$ to obtain the presentation of \Cref{thm: thm A intro} from the introduction.
\end{rem}

\subsection{Applications}\label{ssec: G computation}

We now illustrate how more careful analysis of the topological mirror 
can be used to leverage \cref{cor: K theory as fixed points} to
give explicit formulas. We first
record that, as a further immediate corollary, we get an Atiyah--Hirzebruch
spectral sequence computing the $K$-groups of a torus.

\begin{cor}\label{spectral sequence}
Let $\t$ be an algebraic torus over a field $F$ with absolute Galois group $\GalG$.
The equivalence of \cref{equivariant mainthm} induces a multiplicative
spectral sequence of the form
\[
E^2_{p,q}=H_p^{\GalG}({\topTor{\t}};\pi_qK_{\GalG}(F)) \implies K_{p+q}(T)\,.
\]
\end{cor}
Note that, since the $K_0$ of any field is given by $\ZZ$, the 
rank map $\pi_0K_{\GalG}(\baf)\to \underline{\ZZ}$ to the constant Mackey functor with 
value $\ZZ$ is an equivalence. In particular, since the spectral sequence is concentrated
in positive degrees, we also learn the following.

\begin{cor}\label{cor: K0 homology}
For any torus $\t/\baf$ we have an equivalence
\begin{equation}\label{eq: K0 homology}
H_0^{\GalG}({\topTor{\t}};\underline{\ZZ}) \xrightarrow{\sim} K_0(\t),
\end{equation}
where the left hand side is Bredon homology with coefficients in the constant
Mackey-functor $\underline{\ZZ}$.
\end{cor}

The left hand side of \cref{cor: K0 homology} now admits an explicit description in terms
of generators and relations:
For any scheme $X/\baf$, write $\Pic(X)$ for the Picard group of $X$, i.e.~the group
of $\otimes$-invertible objects in $\rm{QCoh}(X)^\heart$.
If $\baf \to E$ is a finite \'etale extension, then restriction of scalars defines a functor 
$ \rm{Res}_{E/\baf}\colon \rm{QCoh}(X_E) \too \rm{QCoh}(X)$.
Restricting to line bundles on $X$ and applying $K_0$ thus gives a map
\[\rm{Res}_{E/\baf}\colon \Pic(X_E) \too K_0(X).\]

With this in hand, we recover the following computation of $K_0(T)$ due to Merkurjev 
and Panin from \cite{mp97}.

\begin{cor}[{\cite[Thm.~9.1]{mp97}}] \label{K_0 computation}
Let $\t$ be a torus over a field $\baf$ with splitting field $L/\baf$.
The map of abelian groups
  \[ \rm{Res}_{E/F} \colon \bigoplus_{F\subseteq E \subseteq L} \ZZ[\Pic(\t_E)]
  \too K_0(\t),\]
  is surjective, with kernel generated by all expressions of the form
  \[[E':E](\sL) - (\sL \otimes_E E'),\] 
   where $\sL\in \Pic(T_{E})$ and $E\into E'$ is a $\baf$-linear finite \'etale extension.\footnote{not necessarily compatible with the embedding to $L$.}
\end{cor}
\begin{proof}
 The tensor product of $\SS[{\topTor{\t}}]$ and $K_{\GalG}(\baf)$ is computed by
 the coend formula
 \[ K_{\GalG}(\baf)[{\topTor{\t}}]^G \simeq 
 \int^{\GalG/H\in \cO_{\GalG}} ({\topTor{\t}})^H \otimes K_{\GalG}(\baf)^H.\]
Since $\pi_0$ commutes with colimits and tensor products, \cref{cor: K0 homology}
 implies that $K_0(\t)$ is given by
 \[ H_0^G({\topTor{\t}};\underline{\ZZ}) 
 \simeq \int^{\GalG/H\in\cO_{\GalG}} \ZZ \otimes \pi_0{\topTor{\t}}^{H}
 \simeq \left(\bigoplus_{\GalG/H \in \cO_{\GalG}} 
 \ZZ\otimes \pi_0({\topTor{\t}}^{H})\right)/\sim.\]
 We claim that we have natural equivalence
 \[ \pi_0 ( {\topTor{\t}}^{H} )\simeq \Pic(\t_{E})\]
 where $E= \bar{\baf}^{H}$. Indeed, since $\Pic(\t_{\bar{\baf}})$ is trivial, we have by
 Galois descent that
 \[ \Pic(\t_{E})\simeq H^1(H, \sO^{\times}_{\t_{\bar{\baf}}}),\]
where $\sO_{\t_{\bar{\baf}}}$ denotes the ring of functions. Then, there is a canonical
$\GalG$-equivariant decomposition of the units
\[ \sO^{\times}_{\t_{\bar{\baf}}} \simeq \cow{\t_{\bar{\baf}}} \times \bar{\baf}^{\times} 
=\Lambda \times \bar{\baf}^{\times}.\]
By Hilbert 90, we have $H^1(H,\bar{\baf}^{\times})= 0$. Putting everything together
and recalling that ${\topTor{\t}}\simeq \deloop{\Lambda}$, we obtain equivalences
\[ \Pic(\t_{E}) \simeq H^1(H, \Lambda) \simeq \pi_0({\topTor{\t}}^{H}) \]
as desired. Thus, we have obtained an equivalence 
\[ K_0(\t) \simeq \left(\bigoplus_{\GalG/H\in \cO_{\GalG}} \ZZ[\Pic(\bar{\baf}^{H})]\right)/\sim\]
and leave it to the reader to verify that the relations enforced by the coend formula are the
correct ones.
\end{proof}
As a further application, we recover a computation of the $K$-theory of tori of rank $1$. From now on, assume that $\baf$ is of characteristic different from $2$. 
Every such torus is of the form 
\[
T_a := \{(x,y) |: x^2 - ay^2 = 1\} 
\]
for $a\in \baf^\times$, and two such tori $T_a$ and $T_{a'}$ are isomorphic if and only if $a \equiv a' \mod (\baf^\times)^2$. 
The character lattice of $T_a$ is then a single copy of $\ZZ$ with the action of the Galois group being the inflation of the sign representation of $C_2$ along the map $\GalG \to C_2$, $\gamma \mapsto \frac{\gamma(\sqrt{a})}{\sqrt{a}}$. In other words, it is the homomorphism that classifies the quadratic extension $\baf[\sqrt{a}]/\baf$. 

Since all rank $1$ tori over $\baf$ can be written as affine quadrics, the computation of their $K$-theory is a (very) special case of Swan's computation of the $K$-theory of projective and affine quadratic hypersurfaces \cite{swan1985k}. 

\begin{cor}[\cite{swan1985k}]
Let $\baf$ be a field of characteristic different from $2$, $a\in \baf^\times$ and set 
\[T_a \coloneq \Spec{\baf[x,y]/(x^2-ay^2 -1)}.\] 
We have an equivalence of spectra
\[
K(T_a) \simeq K(\baf)\oplus \mathrm{cofib}(K(\baf[\sqrt{a}])\xrightarrow{\rho} K(\baf)),
\]
where $\rho$ is the map induced by restriction of scalars
along the field extension $\baf \to \baf[\sqrt{a}]$.
\end{cor}
\begin{proof}
Let $\GalG$ denote the absolute Galois group of $\baf$. 
 The topological mirror $\topTor{T_a}$ is a circle $S^1$, and the action of 
 $\GalG$ on this circle is given by the composition 
\[
\GalG \to C_2 \to \Aut(S^1)
\]
where the first map is the one described above (classifying $\baf[\sqrt{a}]$) and the
second map is the (geniune) action of $C_2$ on $\topTor{T_a}$ by reflection along an axis. 
The space $\topTor{T_a}$ has an equivariant cell structure with two $0$-cells $x$ and $y$ 
(the fixed points of the $C_2$-action) and two 1-cells $e_1$ and $e_2$, the upper and 
lower semicircles, which get switched by the $C_2$-action, see \Cref{figure}.
6
\begin{figure}[hbt!]
\begin{center}
\begin{tikzpicture}[scale=1.4, line cap=round, line join=round]
  \draw (0,0) circle (1);
  \fill (-1,0) circle (1.2pt);
  \fill ( 1,0) circle (1.2pt);
  \node[below] at (-1.1,0) {$x$};
  \node[below] at ( 1.1,0) {$y$};
  \node at (0,  1.2) {$e_1$};
  \node at (0, -1.2) {$e_2$};
  \node at (0.28, 0.1) {$\sigma$};
  \coordinate (A) at (0.2,  0.95);
  \coordinate (B) at (0.2, -0.95);
  \draw[<->, thick]
    (A) .. controls (0.5, 0.4) and (0.5, -0.4) .. (B);
 \draw[dotted](-2,0)--(2,0);
\end{tikzpicture}
\end{center}
\caption{The equivariant cell structure of $\topTor{\t_a}$}
\label{figure}
\end{figure}

We thus have a cell-attachment cofiber sequence 
\[
\GalG/\GalG_0 \otimes \Sph_\GalG  \too \Sph_\GalG \oplus \Sph_\GalG  \too \Sph_\GalG \otimes \topTor{T_a} \in \Sp_{\GalG}. 
\]
Tensoring with the equivariant $K$-theory spectrum of $F$ we get a cofiber sequence 
\[
K_G(F) \otimes \GalG/\GalG_0 \too K_G(F) \oplus K_G(F) \too K_G(F)[\topTor{T-a}] \in \Sp_{\GalG}. 
\]
By \cref{equivariant mainthm}, the $\GalG$-fixed points of the last term are given by $K(T_a)$. Since taking fixed points preserves cofiber sequences and 
\[ (K_{\GalG}(\baf) \otimes \GalG/\GalG_0)^\GalG \simeq (K_G(\baf))^{\GalG_0} 
    \simeq K(\baf[\sqrt{a}]), \]
we obtain a cofiber sequence of $\GalG$-fixed points spectra 
\[ K(\baf[\sqrt{a}]) \too K(F)\oplus K(F) \too K(T_a). \]
It remains to identify the left hand map.\\
This map is induced from the attaching map of the $1$-cells to the $0$-cells. Since the $1$-cell is induced from $\GalG_0$, the data of this attaching map is the same as that of the $\GalG_0$-equivariant attaching map of $e_1$ to $\{x\} \sqcup \{y\}$. This is the attaching map of an interval to its endpoints, so the map $\Sph \to \Sph \oplus \Sph$ we get is $(1,-1)$. Passing back to the $\GalG$-equivariant maps, we deduce that the first map is given by 
\[
(\tr_{\GalG_0}^\GalG, - \tr_{\GalG_0}^\GalG)\colon K(\baf[\sqrt{a}]) \to K(F), 
\]
where $\tr^{\GalG_0}_\GalG$ is the transfer map from the $\GalG$-fixed points to the $\GalG$-fixed points. By definition of $K^F$, this map is given by restriction of scalars along $\baf[\sqrt{a}] \to \baf$. Finally, performing the change of coordinates $(a,b)\mapsto (a,a+b)$ on the middle term, we see that we have a cofiber sequence 
\[ K(\baf[\sqrt{a}]) \oto{(\mathrm{res}^{\baf[\sqrt{a}]}_{\baf},0)} K(F)\oplus K(F)
\too K(T_a), \]
which proves the claim.
\end{proof}

\printbibliography{}
\end{document}